\documentclass[12pt]{article}
\usepackage{eurosym}
\usepackage{epsfig}
\usepackage{graphicx}
\usepackage{color}
\usepackage{amsmath,amsfonts,amssymb,amsthm}

\newtheorem{thrm}{Theorem}

\def\dt{\Delta t}
\newfont{\afont}{cmr12 at 8pt}
\newfont{\bfont}{cmr12 at 7pt}
in \evensidemargin 0.25true in \topmargin -0.2true in \headsep
0.4true in
\def\vec#1{\leqavevmode\vtop{\hbox to 1em{\hss$#1$\hss}
\vskip-\baselineskip\vskip1.1ex\hbox to
1em{\hfil$\scriptstyle\sim$\hfil}}}
\begin{document}

\title{Numerical solutions of the generalized equal width wave equation using Petrov Galerkin method}
\author{Samir Kumar Bhowmik$^1$ and Seydi Battal Gazi Karakoc$^2$ \\
$1.$ Department of Mathematics, University of Dhaka\\
Dhaka 1000, Bangladesh. \\
e-mail: bhowmiksk@gmail.com \\
$2.$ Department of Mathematics, Faculty of Science and Art, \\
Nevsehir Haci Bektas Veli University, Nevsehir, 50300, Turkey.\\
e-mail: sbgkarakoc@nevsehir.edu.tr \\
}
\maketitle

\begin{abstract}
In this article we consider a generalized equal width wave (GEW) equation which is
a significant nonlinear wave equation as it can be used to model 
many problems occurring in applied sciences.   As the analytic solution of the (GEW) equation of this kind
can be obtained hardly, developing  numerical
solutions for this type of  equations is of enormous importance and interest. 
Here we are interested in a
Petrov-Galerkin method, in which element shape functions are quadratic and
weight functions are linear B-splines.
We firstly investigate the existence and uniqueness of solutions of the weak form of
the equation. Then we establish the theoretical bound of the error in
the semi-discrete spatial scheme as well as of a full discrete scheme at $t=t^{n}$.
Furthermore, a powerful Fourier analysis has been applied to show that the proposed scheme
is unconditionally stable. Finally, propagation of single and double
solitary waves and evolution of solitons are analyzed to demonstrate the
efficiency and applicability of the proposed numerical scheme by calculating
the error norms (in $L_{2}(\Omega)$ and $L_{\infty}(\Omega)$).
The three invariants ($%
I_{1},I_{2}$ and $I_{3})$ of motion have been commented to verify the
conservation features of the proposed algorithms. Our proposed numerical scheme has been
compared with other published schemes and demonstrated to be valid, effective and it outperforms the others.
\end{abstract}

%
%
%




\section{Introduction}

Nonlinear partial differential equations are extensively used to explain
complex phenomena in different fields of science, such as plasma physics,
fluid mechanics, hydrodynamics, applied mathematics, solid state physics and
optical fibers. One of the important issues to nonlinear partial
differential equations is to seek for exact solutions. Because of the
complexity of nonlinear differential equations, exact solutions of these
equations are commonly not derivable. Owing to the fact that only limited
classes of these equations are solved by analytical means, numerical
solutions of these nonlinear partial differential equations are very
functional to examine physical phenomena. The regularized long wave (RLW)
equation,

\begin{equation}
U_{t}+U_{x}+\varepsilon UU_{x}-\mu U_{xxt}=0,  \label{rlw}
\end{equation}%
is a symbolisation figure of nonlinear long wave and can define many
important physical phenomena with weak nonlinearity and dispersion waves,
including nonlinear transverse waves in shallow water, ion-acoustic and
magneto hydrodynamic waves in plasma, elastic media, optical fibres,
acoustic-gravity waves in compressible fluids, pressure waves in liquid--gas
bubbles and phonon packets in nonlinear crystals \cite{mei}. The RLW
equation was first suggested to describe the behavior of the undular bore by
Peregrine \cite{pereg,pereg1}, who constructed the first numerical method of
the equation using finite difference method. RLW equation is an alternative
description of nonlinear dispersive waves to the more usual

\begin{equation}
U_{t}+\varepsilon UU_{x}+\mu U_{xxx}=0,  \label{KdV}
\end{equation}%
Korteweg-de Vries (KdV) equation \cite{ben}. This equation was first
generated by Korteweg and de Vries to symbolise the action of one
dimensional shallow water solitary waves \cite{khalid}. The equation has
found numerous applications in physical sciences and engineering field such
as fluid and quantum mechanics, plasma physics, nonlinear optics, waves in
enharmonic crystals, bubble liquid mixtures, ion acoustic wave and
magneto-hydrodynamic waves in a warm plasma as well as shallow water waves.
The Equal Width (EW) wave equation

\begin{equation}
U_{t}+\varepsilon UU_{x}-\mu U_{xxt}=0,  \label{ew}
\end{equation}%
which is less well recognised and was introduced by Morrison et al. \cite%
{morrison} is a description alternative to the more common KdV and RLW
equations. This equation is named equal width equation, because the
solutions for solitary waves with a perpetual form and speed, for a given
value of the parameter $\mu $, are waves with an equal width or vawelength
for all wave amplitudes \cite{hmd}. The solutions of this equation are sorts
of solitary waves called as solitons whose figures are not changed after the
collision. GEW equation, procured for long waves propagating in the positive
$x$ direction takes the form

\begin{equation}
U_{t}+\varepsilon U^{p}U_{x}-\mu U_{xxt}=0,  \label{gew}
\end{equation}%
where $p$ is a positive integer, $\varepsilon $ and $\mu $ are positive
parameters, $t$ is time and $x$ is the space coordinate, $U(x,t)$ is the
wave amplitude. Physical boundary conditions require $U\rightarrow 0$ as $%
\left\vert x\right\vert \rightarrow \infty $. For this work, boundary and
initial conditions are chosen

\begin{equation}
\begin{array}{l}
U(a,t)=0,~~~~~~\ \ \ \ \ U(b,t)=0, \\
U_{x}(a,t)=0,~~~~\ \ \ ~~U_{x}(b,t)=0, \\
U_{xx}(a,t)=0,~~~~\ \ \ ~~U_{xx}(b,t)=0, \\
U(x,0)=f(x),~~\ \ \ \ a\leq x\leq b,%
\end{array}
\label{Bou.Con.}
\end{equation}%
where $f(x)$ is a localized disturbance inside the considered interval and
will be designated later. In the fluid problems as known, the quantity $U$ \
is associated with the vertical displacement of the water surface but in the
plasma applications, $U$ is the negative of the electrostatic potential.
That's why, the solitary wave solution of Eq.$(\ref{gew})$ helps us to find
out the a lot of physical phenomena with weak nonlinearity and dispersion
waves such as nonlinear transverse waves in shallow water, ion-acoustic and
magneto- hydrodynamic waves in plasma and phonon packets in nonlinear
crystals \cite{sbgk}. The GEW equation which we tackle here is based on the
EW equation and relevant to the both generalized regularized long wave
(GRLW) equation \cite{kaya,kaya1} and the generalized Korteweg-de Vries
(GKdV) equation \cite{gard4}. These general equations are nonlinear wave
equations with $(p+1)$th nonlinearity and have solitary wave solutions,
which are pulse-like. The investigate of GEW equation ensures the
possibility of investigating the creation of secondary solitary waves and/or
radiation to get insight into the corresponding processes of particle
physics \cite{dodd,levis}. This equation has many implementations in
physical situations for example unidirectional waves propagating in a water
channel, long waves in near-shore zones, and many others \cite{panahipour}.
If $p=1$ is taken in Eq.$(\ref{gew})$ the EW equation [15-20] is obtained
and if $p=2$ is taken in Eq.$(\ref{gew}),$ the obtained equation is named as
the modified equal width wave (MEW) equation [21-27]. In recent years,
various numerical methods have been improved for the solution of the GEW
equation. Hamdi et al. \cite{hmd} generated exact solitary wave solutions of
the GEW equation. Evans and Raslan \cite{evans} investigated the GEW
equation by using the collocation method based on quadratic B-splines to
obtain the numerical solutions of the single solitary wave, interaction of
solitary waves and birth of solitons. The GEW equation solved numerically by
a B-spline collocation method by Raslan \cite{raslan}. The homogeneous
balance method was used to construct exact travelling wave solutions of
generalized equal width equation by Taghizadeh et al. \cite{tag}. The
equation is solved numerically by a meshless method based on a global
collocation with standard types of radial basis functions (RBFs) by \cite%
{panahipour}. Quintic B-spline collocation method with two different
linearization techniques and \ a lumped Galerkin method based on
B-spline functions were employed to obtain the numerical solutions of the
GEW equation by Karakoc and Zeybek, \cite{sbgk,sbgk6} respectively. Roshan
\cite{roshan}, applied Petrov-Galerkin method using the linear hat function
and quadratic B-spline functions as test and trial function respectively for
the GEW equation.

In this study, we have constructed a lumped Petrov-Galerkin method for the
GEW equation using quadratic B-spline function as element shape function and
linear B-spline function as the weight function. Context of this work has
been planned as follows:  
\begin{itemize}
\item[-] A semi-discrete Galerkin finite element scheme of the equation along with the
error bounds are demonstrated in Section~2.
\item[-] A full discrete Galerkin finite element scheme has been studied in Section~3. 
\item[-] Section~4 is concerned with the construction and implementation of the
Petrov-Galerkin finite element method to the GEW equation. 
\item[-] Section~5 contains a linear stability analysis of the scheme.
\item[-] Section~6 includes analysis of the motion of single solitary wave,
         interaction of two solitary wave and evolution of solitons with different
         initial and boundary conditions. 
\item[-] Finally, we conclude the study with some remarks on this study.
\end{itemize}
\section{Variational formulation and its analysis}

The higher order nonlinear initial boundary value problem \eqref{gew} can be
written as
\begin{equation}
u_{t}-\mu \Delta u_{t}=\nabla \mathcal{F}(u),~~~~~  \label{grlw11}
\end{equation}%
where $ \mathcal{F}(u) = \frac{1}{p+1} u^{p+1}, $ subject to the initial
condition
\begin{equation}
u(x,0)=f_1(x),~~~~~a\leq x\leq b,  \label{intl}
\end{equation}%
and the boundary conditions
\begin{equation}
\begin{array}{lll}
u(a,t)=0,~~~~~u(b,t)=0, &  &  \\
u_{x}(a,t)=0,~~~~~u_{x}(b,t)=0, &  &  \\
u_{xx}(a,t)=0,~~~~~u_{xx}(b,t)=0, & t>0. &
\end{array}
\label{bndry}
\end{equation}

To define the weak form of the solutions of \eqref{grlw11} and to
investigate the existence and uniqueness of solutions of the weak form we
define the following spaces. Here $H^k(\Omega)$, $k\ge 0$ (integer) is
considered as an usual normed space of real valued functions on $\Omega$ and
\begin{equation*}
H_0^{k}(\Omega) = \left\{v \in H^k(\Omega): D^{i}v = 0\ \text{on }
\partial\Omega,\ i = 0, 1, \cdots, k-1 \right\}
\end{equation*}
where $D = \frac{\partial}{\partial x}$. We denote the norm on $H^k(\Omega)$
by $\|\cdot \|_k$ which is the well known usual $H^k$ norm, and when $k=0$, $%
\|\cdot \|_0 = \|\cdot \|$ represents $L_2$ norm and $(\cdot, \cdot)$
represents the standard $L_2$ inner product~\cite{Noureddine2013,
ThomeeVidar2006}.

Multiplying \eqref{grlw11} by $\xi\in H_0^1(\Omega)$, and then integrating
over $\Omega$ we have
\begin{equation*}
(u_{t}, \xi) - \mu (\Delta u_{t}, \xi) = (\nabla \mathcal{F}(u), \xi).
\end{equation*}
Applying Green's theorem for integrals on the above continuous inner
products we aim to find $u(\cdot, t)\in H_0^1(\Omega)$ so that
\begin{equation}  \label{bbmbur_v2}
\left(u_{t}, \xi \right) + \mu \left(\nabla u_{t}, \nabla\xi \right) =
-\left( \mathcal{F}(u), \nabla \xi \right), \ \forall \ \xi\in H_0^1(\Omega),
\end{equation}
with $u(0) = u_0$. Here we state the uniqueness theorem without proof which
can be well established following \cite{Noureddine2013, ThomeeVidar2006}.

\begin{thrm}
\label{thrm01} If $u$ satisfies \eqref{bbmbur_v2} then
\begin{equation*}
\|u(t)\|_1 = \|u_0\|_1,\ t\in\ (0,\ T],\ \text{and }\
\|u\|_{L^\infty(L^\infty(\Omega))} \le C\|u_0\|_1
\end{equation*}
holds if $u_0\in H_0^1(\Omega)$, and $C$ is a positive constant.
\end{thrm}

\begin{thrm}
Assume that $u_0 \in H_0^1(\Omega)$ and $T >0$. Then there exists one and
only one $u\in H_0^1(\Omega)$ satisfying  \eqref{bbmbur_v2} for any $T >0$
such that
\begin{equation*}
u \in L^\infty(0, T, H_0^1(\Omega)) \ \text{with }\ (u(x, 0), \xi) = (u_0,
\xi), \xi\in H_0^1(\Omega).
\end{equation*}
\end{thrm}

%
%

\subsection{Semi-discrete Galerkin Scheme}

For any $0<h<1$ let $S_{h}$ of $H_{0}^{1}(\Omega )$ be a finite dimensional
subspace such that for $u\in H_{0}^{1}(\Omega)\cap H^{3}(\Omega )$, $\exists$
a constant $C$ independent of $h$~\cite{Noureddine2013, ThomeeVidar2006,
ciarlet} such that
\begin{equation}
\inf_{\xi \in S_{h}}\Vert u-\xi \Vert \leq Ch^{3}\|u\|_3.  \label{interp01}
\end{equation}%
Here it is our moto to look for solutions of a semi-discrete finite element
formulation of \eqref{grlw11} $u_{h}:[0,\ T]\rightarrow S_{h}$ such that
\begin{equation}
\left( u_{ht},\xi \right) + \left( \nabla u_{ht},\nabla \xi \right) =-\left(
\mathcal{F} (u_{h}),\nabla \xi \right) ,\ \ \forall\ \xi \in S_{h},
\label{bbmbur_v51}
\end{equation}
where $u_{h}(0)=u_{0,h}\in S_{h}$ approximates $u_{0}$. We start here first
by stating a priori bound of the solution of \eqref{bbmbur_v51} below before
establishing the original convergence result.

\begin{thrm}
\label{thrm04} Let $u_h \in S_h$ be a solution of \eqref{bbmbur_v51}. Then $%
u_h \in S_h$ satisfies
\begin{equation*}
\|u_h\|_1^2 = \|u_{0,h}\|_1^2,\ t\in\ (0,\ T],\
\end{equation*}
\text{and }\
\begin{equation*}
\|u_h\|_{L^\infty(L^\infty(\Omega))} \le C\|u_{0,h}\|_1
\end{equation*}
holds where $C$ is a positive constant.
\end{thrm}

\begin{proof}
  The proof is trivial, it follows from \cite{Battel_SKB_2018}.
\end{proof}
Our next goal is to establish the theoretical estimate of the error in the
semi-discrete scheme \eqref{bbmbur_v51} of \eqref{bbmbur_v2}. To that end
here we start by considering the following bilinear form
\begin{equation*}
\mathcal{A}(u,v)=(\nabla u,\nabla v),\ \forall\ u,\ v\in H_{0}^{1}(\Omega),
\end{equation*}%
which satisfies the boundedness property
\begin{equation}
|\mathcal{A}(u,v)|\leq M\Vert u\Vert _{1}\Vert v\Vert _{1},\forall \ u,\
v\in H_{0}^{1}(\Omega)  \label{boundedness}
\end{equation}%
and coercivity property (on $\Omega $)
\begin{equation}
\mathcal{A}(u,u)\geq \alpha \Vert u\Vert _{1},\forall \ u\in H_{0}^{1}
(\Omega),\ \text{for some }\alpha \in \mathbb{R}.  \label{coercivity}
\end{equation}%
Let $\tilde{u}$ be an auxiliary projection of $u$ \cite{Noureddine2013,
ciarlet, ThomeeVidar2006}, then $\mathcal{A}$ satisfies
\begin{equation}
\mathcal{A}(u-\tilde{u},\xi )=0,\ \xi \in S_{h}.  \label{projection}
\end{equation}

Now the rate of convergence (accuracy) in such a spatial approximation %
\eqref{bbmbur_v51} of \eqref{bbmbur_v2} is given by the following theorem.

\begin{thrm}
Let $u_h\in S_h$ be a solution of \eqref{bbmbur_v51} and $u\in H_0^1(\Omega)$
be that of \eqref{bbmbur_v2}, then the following inequality holds
\begin{equation*}
\|u - u_h\|\le C h^3,
\end{equation*}
where $C>0$ if $\|u(0) - u_{0, h}\|\le Ch^3$ holds.
\end{thrm}

%
\begin{proof}
Letting
$
   \mathcal{E} = u-u_h  = \psi +\theta,
$
    where  $ \psi = u - \tilde u$ and  $\theta = \tilde u - u_h$,
we write
\begin{align*}
\alpha \|u-\tilde u\|_1^2 &\le  \mathcal{A}(u-\tilde u, u-\tilde u) \\
                           & =  \mathcal{A}(u-\tilde u, u-\xi),\ \xi\in S_h. 
\end{align*}
From \eqref{boundedness}, \eqref{projection} and \cite{ThomeeVidar2006}  it follows that
\begin{equation}\label{pro_bound}
  \|u-\tilde u\|_1 \le \inf_{\xi\in S_h}\| u - \xi\|_1,
\end{equation}
and thus \eqref{interp01} and \eqref{pro_bound} confirms the following inequalities
\[
 \|\psi\|_1 \le C h^2 \|u\|_3, \ \text{and  so   \  }\ \|\psi \| \le  C h^3 \|u\|_3.
\]
Now applying $\frac{\partial}{\partial t}$ on \eqref{projection} and having some simplifications yields~\cite{ThomeeVidar2006}
\[
   \|\psi_t\|\le  C h^3 \|u_t\|_3.
\]
Also we subtract  \eqref{bbmbur_v51} from \eqref{bbmbur_v2} to obtain
\begin{equation}\label{bbmbur_v544}
  (\theta_t, \xi) + (\nabla \theta_t, \nabla\xi)   =  (\psi_t, \xi) - (\mathcal{F}(u)-\mathcal{F}(u_h),\nabla \xi).
\end{equation}
Now we substitute  $\xi = \theta$ in \eqref{bbmbur_v544}, and then apply Cauchy-Schwarz inequality  to obtain
\[
 \frac{1}{2}\frac{d}{dt} \|\theta\|_1^2  \le    \|\psi_t\|\|\theta\|  +\|\mathcal{F}(u) - \mathcal{F}(u_h)\| \|\nabla \theta\|.
\]
Here
\[
 \|\mathcal{F}(u) - \mathcal{F}(u_h)\| \le C(\|\psi\| + \|\theta\|),
\]
comes from Lipschitz conditions of $\mathcal{F}$ and boundedness of $u$ and $u_h$.  Thus
\[
 \frac{d}{dt} \|\theta\|_1^2 \le   C\left( \|\psi_t\|^2 +  \|\psi\|^2  + \|\theta\|^2 + \|\nabla \theta\|^2 \right).
\]
So
\[
 \|\theta\|_1^2 \le \|\theta(0)\|_1^2+  C\int_0^t \left( \|\psi_t\|^2 +  \|\psi\|^2  + \|\theta\|^2 + \|\nabla \theta\|^2 \right)dt.
\]
Hence Gronwall's lemma, bounds of $\psi$ and $\psi_t$ confirms
\[
  \|\theta\|_1 \le C(u) h^3,
\]
if $\theta(0) = 0$, completes the proof~\cite{ThomeeVidar2006, ciarlet}.
\end{proof}

\section{Full discrete scheme}

Here we aim to find solution of the semi-discrete problem \eqref{bbmbur_v51}
over $[0, T]$, $T>0$. Let $N$ be a positive full number and $\Delta t =
\frac{T}{N}$ so that $t^n = n\Delta t$, $n = 0,\ 1,\ 2,\ 3,\cdots,\ N.$ Here
we consider
\begin{equation*}
\phi^n = \phi(t^n),\ \quad \phi^{n-1/2} = \frac{\phi^n + \phi^{n-1}}{2}\quad
\& \quad \ \partial_t\phi^n = \frac{\phi^n - \phi^{n-1}}{\Delta t }.
\end{equation*}
Using the above notations we present a time discretized finite element
Galerkin scheme by
\begin{equation}
\left( \partial_t U^n,\xi \right) + \left( \nabla \partial_t U^n ,\nabla \xi
\right) = -\left( \mathcal{F} (U^{n-1/2}),\nabla \xi \right) ,\ \xi \in
S_{h},  \label{bbmbur_v5}
\end{equation}
where $U^0 = u_{0, h}$.

\begin{thrm}
If $U^n$ satisfies \eqref{bbmbur_v5} then
\begin{equation*}
\|U^J\|_1 = \|U^0\|_1\ \text{ for all }\ 1\le J\le N
\end{equation*}
and there exists a positive constant $C$ such that
\begin{equation*}
\|U^J\|_\infty \le C \|U^0\|_1\ \text{ for all} \ 1\le J\le N.
\end{equation*}
\end{thrm}

\begin{proof}
Substituting $\xi = U^{n-1/2}$ in \eqref{bbmbur_v5}  it is easy to see that
\begin{equation}\label{bbmbur_v11}
 \partial_t \left(\|U^n\|^2 + \|\nabla U^n\|^2  \right) = - \left(\mathcal{F}(U^{n-1/2}), \nabla U^{n - 1/2} \right) = 0.
\end{equation}
Thus the proof of the first part  of the theorem follows from  a sum from $n = 1$ to $J$ and that of the second part follows from  the Sovolev embedding theorem~\cite{ThomeeVidar2006}.
\end{proof}
Now we focus on to establishing the theoretical upper bound of the error in
such a full discrete approximation \eqref{bbmbur_v11} at $t = t^n$.

\begin{thrm}
Let $h$ and $\Delta t$ be sufficiently small, then
\begin{equation*}
\|u^j - U^j\|_\infty \le C(u, T) (h^3 + \Delta t^2) \text{ for } 1\le j \le
N \text{ and } u_0^h = \tilde u(0)
\end{equation*}
where $C$ is independent of $h$ and $\Delta t$.
\end{thrm}

\begin{proof}
Let
\begin{align*}
\mathcal{E}^{n} &= u^n - U^n = \psi^n + \theta^n
\end{align*}
where $\psi^n =  u^n - \tilde{u^n}$, $\theta^n = \tilde{u^n} - U^n$, $u^n = u(t^n)$, and $\tilde{u^n} = \tilde u(t^n)$.
From \eqref{bbmbur_v2} and \eqref{bbmbur_v5} along with auxiliary projection defined in the previous section the following equality holds
\begin{equation}\label{bbmbur_v6}
(\partial_t \theta^n,\xi) + (\nabla\partial_t \theta^n, \nabla \xi )  = (\partial_t \psi^n,\xi) + (\tau^n,\xi) + (\nabla \tau^n,\nabla\xi)
+ \left( \mathcal{F}(u^{n-1/2}) - \mathcal{F}(U^{n-1/2}), \nabla \xi \right),
\end{equation}
where $\tau^n  = u^{n-1/2} - \partial_t u^n$.   Now substituting  $\xi$  by  $\theta^{n-1/2}$ in \eqref{bbmbur_v6}  yields
\begin{equation}\label{bbmbur_v7}
\frac{1}{2}\partial_t \|\theta^n\|^2_1  = C \left( \|\partial_t \psi^n\|^2 + \|\tau^n\|_1^2 + \|\theta^{n-1/2}\|_1^2 +
 \left\|\mathcal{F}(u^{n-1/2}) - \mathcal{F}(U^{n-1/2} \right\|^2\right).
\end{equation}
Now
\begin{equation}\label{bbmbur_v8}
  \|\tau^n\|^2 \le C \dt^3 \int_{t_{n-1}}^{t_n} \|u_{ttt}(s)\|^2 ds,
\end{equation}
and  from  boundedness of $\|U^n\|_\infty$ and $\|u^n\|_\infty$ it yields
\begin{equation}\label{bbmbur_v9}
 \left\|\mathcal{F}(u^{n-1/2}) - \mathcal{F}(U^{n-1/2} \right\| =  C\left(\|\theta^{n-1/2} +\|\psi^{n-1/2}\|\|  \right)
\end{equation}
since $\mathcal{F}$ is a Lipschitz function.   Thus  from \eqref{bbmbur_v7}, \eqref{bbmbur_v8} and \eqref{bbmbur_v9}  it follows that
\begin{eqnarray}\label{bbmbur_v10}
\partial_t \|\theta^n\|^2_1  & \le C\|\theta^{n-1/2} \|_1^2 +  C \left( \|\partial_t \psi^n\|^2 + \|\psi^n\|^2  +  \|\psi^{n-1}\|^2 \right. \nonumber\\
        & \qquad         +  \left. \dt^3 \int_{t_{n-1}}^{t_n} \|u_{ttt}(s)\|^2 ds \right).
\end{eqnarray}
So \eqref{bbmbur_v10}  can be simplified as
\begin{align*}
(1-C\dt)\|\theta^n\|^2_1  & \le (1+C\dt)\|\theta^{n-1/2} \|_1^2 +  C \dt \left( \|\partial_t \psi^n\|^2 + \right.\\
                            &\qquad    \left.   \|\psi^n\|^2 +  \|\psi^{n-1}\|^2
                               +   \dt^3 \int_{t_{n-1}}^{t_n} \|u_{ttt}(s)\|^2 ds \right).
\end{align*}
Choosing $\dt>0$  so that $1- C\dt \ge 0$ and summing over $n = 1, (1) , J$, and from the bounds of $\|\psi^n\|$ and   $\|\partial_t \psi^n\|$  yields
\[
  \|\theta^n\|_1 \le C(u, T) (h^3 + \dt^2),
\]
and the rest follows from the triangular inequality  and sobolev embedding theorem~\cite{ThomeeVidar2006, ciarlet}.
\end{proof}

\section{Construction and Implementation of the method}

We take into account a uniformly spatially distributed set of knots $%
a=x_{0}<x_{1}<...<x_{N}=b$ over the solution interval $a\leq x\leq b$ and $%
h=x_{m+1}-x_{m},$ $m=0,1,2,...,N$. For this partition, we shall need the
following quadratic B-splines $\phi _{m}(x)$ at the points $x_{m},$ $%
m=0,1,2,...,N.$ Prenter \cite{prenter} identified following quadratic
B-spline functions $\phi _{m}(x)$, (\emph{m}= $-1(1)$ $N$), at the points $%
x_{m}$ which generate a basis over the interval $[a,b]$ by
\begin{equation}
\begin{array}{l}
\phi _{m}(x)=\frac{1}{h^{2}}\left\{
\begin{array}{ll}
(x_{m+2}-x)^{2}-3(x_{m+1}-x)^{2}+3(x_{m}-x)^{2},~~ & ~x\in \lbrack
x_{m-1},x_{m}), \\
(x_{m+2}-x)^{2}-3(x_{m+1}-x)^{2},~~ & ~x\in \lbrack x_{m},x_{m+1}), \\
(x_{m+2}-x)^{2},~~ & ~x\in \lbrack x_{m+1},x_{m+2}), \\
0~ & ~otherwise.%
\end{array}%
\right.%
\end{array}
\label{3}
\end{equation}%
We search the approximation $U_{N}(x,t)$ to the solution $U(x,t),$ which use
these splines as the trial functions
\begin{equation}
U_{N}(x,t)=\sum_{j=-1}^{N}\phi _{j}(x)\delta _{j}(t),\   \label{4}
\end{equation}%
in which unknown parameters $\delta _{j}(t)$ will be computed by using the
boundary and weighted residual conditions. In each element, using $h\eta
=x-x_{m}$ $(0\leq \eta \leq 1)$ local coordinate transformation for the
finite element $[x_{m},x_{m+1}],$ quadratic B-spline shape functions $(\ref%
{3})$ in terms of $\eta $ over the interval $[0,1]$ can be reformulated as
\begin{equation}
\begin{array}{l}
\phi _{m-1}=(1-\eta )^{2}, \\
\phi _{m}=1+2\eta -2\eta ^{2}, \\
\phi _{m+1}=\eta ^{2}.%
\end{array}
\label{5}
\end{equation}%
All quadratic B-splines, except that $\phi _{m-1}(x),\phi _{m}(x)$ and $\phi
_{m+1}(x)$ are zero over the interval $[x_{m},x_{m+1}].$ Therefore
approximation function $(\ref{4})$ over this element can be given in terms
of the basis functions $(\ref{5})$ as

\begin{equation}
U_{N}(\eta ,t)=\sum_{j=m-1}^{m+1}\delta _{j}\phi _{j}.  \label{6}
\end{equation}%
Using quadratic B-splines $(\ref{5})$ and the approximation function $(\ref%
{6}),$ the nodal values $U_{m}$ and $U_{m}^{\prime }$ at the knot are found
in terms of element parameters $\delta _{m}$ as follows:
\begin{equation}
\begin{array}{l}
U_{m}=U(x_{m})=\delta _{m-1}+\delta _{m}, \\
U_{m}^{\prime }=U^{\prime }(x_{m})=2(\delta _{m}-\delta _{m-1}).%
\end{array}
\label{7}
\end{equation}%
Here weight functions $L_{m}$ are used as linear B-splines. The linear
B-splines $L_{m}$ at the knots $x_{m}$ are identified as \cite{prenter}:

\begin{equation}
\begin{array}{l}
L_{m}(x)=\frac{1}{h}\left\{
\begin{array}{ll}
(x_{m+1}-x)-2(x_{m}-x),~~ & ~x\in \lbrack x_{m-1},x_{m}), \\
(x_{m+1}-x),~~ & ~x\in \lbrack x_{m},x_{m+1}), \\
0~ & ~otherwise.%
\end{array}%
\right.%
\end{array}
\label{lin}
\end{equation}

A characteristic finite interval $[x_{m},x_{m+1}]$ is turned into the
interval $[0,1]$ by local coordinates $\eta $ concerned with the global
coordinates using $h\eta =x-x_{m}$ $(0\leq \eta \leq 1).$ So linear
B-splines $L_{m}$ are given as

\begin{equation}
\begin{array}{l}
L_{m}=1-\eta \\
L_{m+1}=\eta .%
\end{array}
\label{lin1}
\end{equation}%
Using the Petrov-Galerkin method to Eq.$(\ref{gew}),$ we obtain the weak
form of Eq.$(\ref{gew})$ as

\begin{equation}
\int_{a}^{b}L(U_{t}+\varepsilon U^{p}U_{x}-\mu U_{xxt})dx=0.  \label{8}
\end{equation}%
Applying the change of variable $x\rightarrow \eta $ into Eq.$(\ref{8})$
gives rise to

\begin{equation}
\int_{0}^{1}L\left( U_{t}+\frac{\varepsilon }{h}\hat{U}^{p}U_{\eta }-\frac{%
\mu }{h^{2}}U_{\eta \eta t}\right) d\eta =0,  \label{9}
\end{equation}%
where $\hat{U}$ is got to be constant over an element to make the integral
easier. Integrating Eq.$(\ref{9})$ by parts and using Eq.$(\ref{gew})$ leads
to

\begin{equation}
\int_{0}^{1}[L(U_{t}+\lambda U_{\eta })+\beta L_{\eta }U_{\eta t}]d\eta
=\beta LU_{\eta t}|_{0}^{1},  \label{100}
\end{equation}%
where $\lambda =\frac{\varepsilon \hat{U}^{p}}{h}$ and $\beta =\frac{\mu }{%
h^{2}}.$ Choosing the weight functions $L_{m}$ with linear B-spline shape
functions given by $(\ref{lin1})$ and replacing approximation $(\ref{6})$
into Eq.$(\ref{100})$ over the element $[0,1]$ produces

\begin{equation}
\sum_{j=m-1}^{m+1}[(\int_{0}^{1}L_{i}\phi _{j}+\beta L_{i}^{\prime }\phi
_{j}^{\prime })d\eta -\beta L_{i}\phi _{j}^{\prime }|_{0}^{1}~~]\dot{\delta}%
_{j}^{e}+\sum_{j=m-1}^{m+1}(\lambda \int_{0}^{1}L_{i}\phi _{j}^{\prime
}d\eta )\delta _{j}^{e}=0,  \label{11}
\end{equation}%
which can be obtained in matrix form as

\begin{equation}
\lbrack A^{e}+\beta (B^{e}-C^{e})]\dot{\delta}^{e}+\lambda D^{e}\delta
^{e}=0.  \label{120}
\end{equation}
In the above equations and overall the article, the dot denotes
differentiation according to $t$ and $\delta ^{e}=(\delta _{m-1},\delta
_{m},\delta _{m+1},\delta _{m+2})^{T}$ are the element parameters. $%
A_{ij}^{e},B_{ij}^{e},C_{ij}^{e}$ and $D_{ij}^{e}$ are the $2\times 3$
rectangular element matrices represented by
\begin{equation*}
A_{ij}^{e}=\int_{0}^{1}L_{i}\phi _{j}d\eta =\frac{1}{12}\left[
\begin{array}{ccc}
3 & 8 & 1 \\
1 & 8 & 3%
\end{array}%
\right] ,
\end{equation*}

\begin{equation*}
B_{ij}^{e}=\int_{0}^{1}L_{i}^{\prime }\phi _{j}^{\prime }d\eta =\frac{1}{2}%
\left[
\begin{array}{ccc}
1 & 0 & -1 \\
-1 & 0 & 1%
\end{array}%
\right] ,
\end{equation*}

\begin{equation*}
C_{ij}^{e}=L_{i}\phi _{j}^{\prime }|_{0}^{1}=\left[
\begin{array}{ccc}
2 & -2 & 0 \\
0 & -2 & 2%
\end{array}%
\right] ,
\end{equation*}

\begin{equation*}
D_{ij}^{e}=\int_{0}^{1}L_{i}\phi _{j}^{\prime }d\eta =\frac{1}{3}\left[
\begin{array}{ccc}
-2 & 1 & 1 \\
-1 & -1 & 2%
\end{array}%
\right]
\end{equation*}%
where $i$ takes $m,m+1$ and $j$ takes $m-1,m,m+1$ for the typical element $%
[x_{m},x_{m+1}].$ A lumped value for $U$ is attained from ($\frac{%
U_{m}+U_{m+1}}{2}$)$^{p}$ as
\begin{equation*}
\lambda =\frac{\varepsilon }{2^{p}h}(\delta _{m-1}+2\delta _{m}+\delta
_{m+1})^{p}.
\end{equation*}%
Formally aggregating together contributions from all elements leads to the
matrix equation
\begin{equation}
\lbrack A+\beta (B-C)]\dot{\delta}+\lambda D\delta =0,  \label{13}
\end{equation}%
where global element parameters are $\delta =(\delta _{-1},\delta
_{0},...,\delta _{N},\delta _{N+1})^{T}$ and the $A$, $B,C$ and $\lambda D$
matrices are derived from the corresponding element matrices $%
A_{ij}^{e},B_{ij}^{e},C_{ij}^{e}$ and $D_{ij}^{e}.$ Row $m$ of each matrices
has the following form;
\begin{equation*}
\begin{array}{l}
A=\frac{1}{12}\left( 1,11,11,1,0\right) ,B=\frac{1}{3}(-1,1,1,-1,0), \\
C=(0,0,0,0,0), \\
\lambda D=\frac{1}{3}\left( -\lambda _{1},-\lambda _{1}-2\lambda
_{2},2\lambda _{1}+\lambda _{2},\lambda _{2},0\right)%
\end{array}%
\end{equation*}%
where
\begin{equation*}
\lambda _{1}=\frac{\varepsilon }{2^{p}h}\left( \delta _{m-1}+2\delta
_{m}+\delta _{m+1}\right) ^{p},\ \lambda _{2}=\frac{\varepsilon }{2^{p}h}%
\left( \delta _{m}+2\delta _{m+1}+\delta _{m+2}\right) ^{p}.
\end{equation*}%
Implementing the Crank-Nicholson approach $\delta =\frac{1}{2}(\delta
^{n}+\delta ^{n+1})$ and the forward finite difference $\dot{\delta}=\frac{%
\delta ^{n+1}-\delta ^{n}}{\Delta t}$ in Eq.$(\ref{120})$ we get the
following matrix system:
\begin{equation}
\lbrack A+\beta (B-C)+\frac{\lambda \Delta t}{2}D]\delta ^{n+1}=[A+\beta
(B-C)-\frac{\lambda \Delta t}{2}D]\delta ^{n}  \label{14}
\end{equation}%
where $\Delta t$ is time step. Implementing the boundary conditions ($\ref%
{Bou.Con.})$ to the system $(\ref{14})$, we make the matrix equation square.
This system is efficaciously solved with a variant of the Thomas algorithm
but in solution process, two or three inner iterations $\delta ^{n\ast
}=\delta ^{n}+\frac{1}{2}(\delta ^{n}-\delta ^{n-1})$ are also performed at
each time step to cope with the nonlinearity. As a result, a typical member
of the matrix system $(\ref{14})$ \ may be written in terms of the nodal
parameters $\delta ^{n}$ and $\delta ^{n+1}$ as:
\begin{equation}
\begin{array}{l}
\gamma _{1}\delta _{m-1}^{n+1}+\gamma _{2}\delta _{m}^{n+1}+\gamma
_{3}\delta _{m+1}^{n+1}+\gamma _{4}\delta _{m+2}^{n+1}= \\
\gamma _{4}\delta _{m-1}^{n}+\gamma _{3}\delta _{m}^{n}+\gamma _{2}\delta
_{m+1}^{n}+\gamma _{1}\delta _{m+2}^{n}%
\end{array}
\label{15}
\end{equation}%
where
\begin{equation*}
\begin{array}{l}
\gamma _{1}=\frac{1}{12}-\frac{\beta }{3}-\frac{\lambda \Delta t}{6},~~~\ \
\ \ \ \ \ \gamma _{2}=\frac{11}{12}+\frac{\beta }{3}-\frac{3\lambda \Delta t%
}{6}, \\
\gamma _{3}=\frac{11}{12}+\frac{\beta }{3}+\frac{3\lambda \Delta t}{6},~~~\
\ \ \ \ \ \ \gamma _{4}=\frac{1}{12}-\frac{\beta }{3}+\frac{\lambda \Delta t%
}{6}.%
\end{array}%
\end{equation*}%
To start the iteration for computing the unknown parameters, the initial
unknown vector $\delta ^{0}$ is calculated by using Eqs.($\ref{Bou.Con.}).$
Therefore, using the relations at the knots $U_{N}(x_{m},0)=U(x_{m},0)$, $%
m=0,1,2,...,N$ and $U_{N}^{^{\prime }}(x_{0},0)=U^{^{\prime }}(x_{N},0)=0$
related with a variant of the Thomas algorithm, the initial vector $\delta
^{0}$ is easily obtained from the following matrix form
\begin{equation*}
\left[
\begin{array}{cccccc}
1 & 1 &  &  &  &  \\
& 1 & 1 &  &  &  \\
&  &  & \ddots &  &  \\
&  &  &  & 1 & 1 \\
&  &  &  & -2 & 2%
\end{array}%
\right] \left[
\begin{array}{c}
\delta _{-1}^{0} \\
\delta _{0}^{0} \\
\vdots \\
\delta _{N-1}^{0} \\
\delta _{N}^{0}%
\end{array}%
\right] =\left[
\begin{array}{c}
U(x_{0},0) \\
U(x_{1},0) \\
\vdots \\
U(x_{N},0) \\
hU^{^{\prime }}(x_{N},0)%
\end{array}%
\right] .
\end{equation*}

\section{ Stability analysis}

In this section, to show the stability analysis of the numerical method, we
have used Fourier method based on Von-Neumann theory and presume that the
quantity $U^{p}$ in the nonlinear term $U^{p}U_{x}$ of the equation $(\ref%
{gew})$ is locally constant. Substituting the Fourier mode $\delta
_{j}^{n}=g^{n}e^{ijkh}$ where $k$ is mode number and $h$ is element size,
into scheme ($\ref{15})$
\begin{equation}
g=\frac{a-ib}{a+ib},  \label{16}
\end{equation}%
is obtained and where
\begin{equation}
\begin{array}{l}
a=\left( 11+4\beta \right) \cos \left( \frac{\theta }{2}\right) h+\left(
1-4\beta \right) \cos \left( \frac{3\theta }{2}\right) h, \\
b=2\lambda \Delta t[3\sin \left( \frac{\theta }{2}\right) h+\sin \left(
\frac{3\theta }{2}\right) h].%
\end{array}
\label{17}
\end{equation}%
$|g|$ is found $1$ so our linearized scheme is unconditionally stable.

\section{Computational results and discussions}

The objective of this section is to investigate the deduced algorithm using
different test problems relevant to the dispersion of single solitary waves,
interaction of two solitary waves and the evolution\ of solitons. For the
test problems, we have calculated the numerical solution of the GEW equation
for $p=2,3$ and $4$ using the homogenous boundary conditions and different
initial conditions. The $\mathit{L}_{2}$
\begin{equation*}
\mathit{L}_{2}=\left\Vert U^{exact}-U_{N}\right\Vert _{2}\simeq \sqrt{%
h\sum_{J=0}^{N}\left\vert U_{j}^{exact}-\left( U_{N}\right) _{j}\right\vert
^{2}},
\end{equation*}%
and $\mathit{L}_{\infty }$
\begin{equation*}
\ \mathit{L}_{\infty }=\left\Vert U^{exact}-U_{N}\right\Vert _{\infty
}\simeq \max_{j}\left\vert U_{j}^{exact}-\left( U_{N}\right) _{j}\right\vert
.
\end{equation*}%
error norms are considered to measure the efficiency and accuracy of the
present algorithm and to compare our results with both exact values, Eq.$(%
\ref{exct})$, as well as other results in the literature whenever available.
The exact solution of the GEW equation is taken \cite{evans,sbgk6} to be

\begin{equation}
\mathit{U(x,t)}=\sqrt[p]{\frac{c(p+1)(p+2)}{2\varepsilon }\sec h^{2}[\frac{p%
}{2\sqrt{\mu }}(x-ct-x_{0})]}  \label{exct}
\end{equation}%
which corresponds to a solitary wave of amplitude $\sqrt[p]{\frac{c(p+1)(p+2)%
}{2\varepsilon },}$ the speed of the wave traveling in the positive
direction of the $x$-axis is $c$, width $\frac{p}{2\sqrt{\mu }}$ and $x_{0}$
is arbitrary constant. With the homogenous boundary conditions, solutions of
GEW equation possess three invariants of the motion introduced by
\begin{equation}
I_{1}=\int_{a}^{b}U(x,t)dx,~~~~~I_{2}=\int_{a}^{b}[{U^{2}(x,t)+\mu
U_{x}^{2}(x,t)}]dx,~~~~~I_{3}=\int_{a}^{b}{U^{p+2}(x,t)}dx  \label{invrnt}
\end{equation}%
related to mass, momentum and energy, respectively.

\subsection{Propagation of single solitary waves}

For the numerical study in this case, we firstly select $p=2$, $c=0.5$, $%
h=0.1$, $\Delta t$ $=0.2$, $\mu =1$, $\varepsilon =3$ and $x_{0}$ $=30$
through the interval $[0,80]$ to match up with that of previous papers \cite%
{sbgk,sbgk6,roshan}. These parameters represent the motion of a single
solitary wave with amplitude $1.0$ and the program is performed to time $%
t=20 $ over the solution interval. The analytical values of conservation
quantities are $I_{1}$ $=3.1415927$, $I_{2}$ $=2.6666667$ and $%
I_{3}=1.3333333.$ Values of the three invariants as well as $\mathit{L}_{2}$
and $\ \mathit{L}_{\infty }$-error norms from our method have been found and
noted in Table $(\ref{400})$. Referring to Table$(\ref{400}),$ the error
norms $\mathit{L}_{2}$ and $\mathit{L}_{\infty }$ remain less than $%
1.286582\times 10^{-2}$, $8.31346\times 10^{-3}$ and they are still small
when the time is increased up to $t=20$. The invariants $I_{1},I_{2}$, $%
I_{3} $ change from their initial values by less than $9.8\times \ 10^{-6},$
$3.2\times \ 10^{-5}$ and $1.3\times \ 10^{-5},$ respectively, throughout
the simulation. Also, this table confirms that the changes of the invariants
are in agreement with their exact values. So we conclude that our method is
sensibly conservative. Comparisons with our results with exact solution as
well as the calculated values in \cite{sbgk,sbgk6,roshan} have been made and
showed in Table$(\ref{401})$ at $t=20$. This table clearly shows that the
error norms got by our method are marginally less than the others. The
numerical solutions at different time levels are depicted in Fig. $(\ref{500}%
).$ This figure shows that single soliton travels to the right at a constant
speed and conserves its amplitude and shape with increasing time
unsurprisingly. Initially, the amplitude of solitary wave is $1.00000$ and
its top position is pinpionted at $x=30$. At $t=20,$ its amplitude is noted
as $0.999416$ with center $x=40$. Thereby the absolute difference in
amplitudes over the time interval $[0,20]$ are observed as $5.84\times
10^{-4}$. The quantile of error at discoint times are depicted in Fig.$(\ref%
{501})$ . The error aberration varies from $-8\times 10^{-2}$ to $1\times
10^{-2}$ and the maximum errors happen around the central position of the
solitary wave.
\begin{table}[h!]
\caption{Invariants and errors for single solitary wave with $p=2,$ $c=0.5,$
$h=0.1,$ $\protect\varepsilon =3,$ $\Delta t=0.2,$ $\protect\mu =1,$ $x\in %
\left[ 0,80\right] .$ }
\label{400}\vskip-1.cm
\par
\begin{center}
{\scriptsize
\begin{tabular}{cccccc}
\hline\hline
$Time$ & $I_{1}$ & $I_{2}$ & $I_{3}$ & $L_{2}$ & $L_{\infty }$ \\ \hline
0 & 3.1415863 & 2.6682242 & 1.3333283 & 0.00000000 & 0.00000000 \\
5 & 3.1415916 & 2.6682311 & 1.3333406 & 0.00395289 & 0.00294851 \\
10 & 3.1415934 & 2.6682352 & 1.3333413 & 0.00704492 & 0.00473785 \\
15 & 3.1415948 & 2.6682434 & 1.3333413 & 0.00995547 & 0.00651735 \\
20 & 3.1415961 & 2.6682568 & 1.3333413 & 0.01286582 & 0.00831346 \\
\hline\hline
\end{tabular}
}
\end{center}
\end{table}
\begin{table}[h!]
\caption{Comparisons of results for single solitary wave with $p=2,$ $c=0.5,$
$h=0.1,$ $\protect\varepsilon =3,$ $\Delta t=0.2,$ $\protect\mu =1,$ $x\in %
\left[ 0,80\right] $ at $t=20.$}
\label{401}\vskip-1.cm
\par
\begin{center}
{\scriptsize
\begin{tabular}{cccccc}
\hline\hline
$Method$ & $I_{1}$ & $I_{2}$ & $I_{3}$ & $L_{2}$ & $L_{\infty }$ \\ \hline
Analytic & 3.1415961 & 2.6666667 & 1.3333333 & 0.00000000 & 0.00000000 \\
Our Method & 3.1415916 & 2.6682568 & 1.3333413 & 0.01286582 & 0.00831346 \\
Cubic Galerkin\cite{sbgk} & 3.1589605 & 2.6902580 & 1.3570299 & 0.03803037 &
0.02629007 \\
Quintic Collocation First Scheme\cite{sbgk6} & 3.1250343 & 2.6445829 &
1.3113394 & 0.05132106 & 0.03416753 \\
Quintic Collocation Second Scheme\cite{sbgk6} & 3.1416722 & 2.6669051 &
1.3335718 & 0.01675092 & 0.01026391 \\
Petrov-Galerkin\cite{roshan} & 3.14159 & 2.66673 & 1.33341 & 0.0123326 &
0.0086082 \\ \hline\hline
\end{tabular}
}
\end{center}
\end{table}
{\normalsize
\begin{figure}[h!]
{\normalsize {\scriptsize {\centering{\small %
\includegraphics[scale=.25]{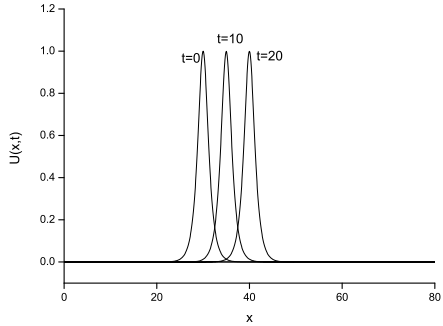}} } }  }
\caption{Motion of single solitary wave for $p=2$, $c=0.5$, $h=0.1$, $\Delta
t$ $=0.2,$ $\protect\varepsilon =3,$ $\protect\mu =1,$ over the interval $%
[0,80]$ at $t=0,10,20.$}
\label{500}
\end{figure}
\begin{figure}[h!]
{\normalsize {\scriptsize {\centering{\small %
\includegraphics[scale=.25]{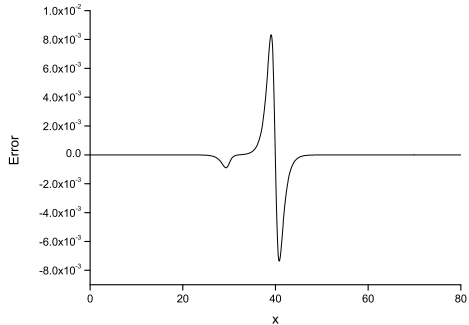}} } }  }
\caption{Error graph for $p=2,$ $c=0.5,$ $h=0.1,$ $\protect\varepsilon =3,$ $%
\Delta t=0.2,$ $\protect\mu =1,$ $x\in \left[ 0,80\right] $ at $t=20$.}
\label{501}
\end{figure}
}

For our second experiment, we take the parameters $p=3,$ $c=0.3,$ $h=0.1,$ $%
\Delta t=0.2,$ $\varepsilon =3,$ $\mu =1$, $x_{0}$ $=30$ with interval $%
[0,80]$ to coincide with that of previous papers \cite{sbgk,sbgk6,roshan}.
Thus the solitary wave has amplitude $1.0$ and the computations are carried
out for times up to $t=20.$ The values of the error norms $L_{2},$ $%
L_{\infty }$ and conservation quantities $I_{1},I_{2}$,$I_{3}$ are found and
tabulated in Table $(\ref{402})$. According to Table$(\ref{402})$ the error
norms $\mathit{L}_{2}$ and $\mathit{L}_{\infty }$ remain less than $%
4.48357\times 10^{-3}$, $3.37609\times 10^{-3}$ and they are still small
when the time is increased up to $t=20$ and the invariants $I_{1},I_{2}$,$%
I_{3}$ change from their initial values by less than $1.78\times \ 10^{-5},$
$2.52\times \ 10^{-5}$, $3.55\times \ 10^{-5},$ respectively. Therefore we
can say our method is satisfactorily conservative. In Table$(\ref{403})$ the
performance of the our new method is compared with other methods \cite%
{sbgk,sbgk6,roshan} at $t=20$. It is observed that errors of the method \cite%
{sbgk,sbgk6,roshan} are considerably larger than those obtained with the
present scheme. The motion of solitary wave using our scheme is graphed at
time $t=0,10,20$ in Fig.$(\ref{502}).$ As seen, single solitons move to the
right at a constant speed and preserves its amplitude and shape with
increasing time as anticipated. The amplitude is $1.00000$ at $t=0$ and
located at $x=30$, while it is $0.999522$ at $t=20$ and located at $x=36$.
Therefore the absolute difference in amplitudes over the time interval $%
[0,20]$ are found as $4.78\times 10^{-4}$. The aberration of error at
discrete times are drawn in Fig.$(\ref{503}).$ The error deviation varies
from $-3\times 10^{-3}$ to $4\times 10^{-3}$ and the maximum errors arise
around the central position of the solitary wave.
\begin{table}[h!]
\caption{Invariants and errors for single solitary wave with $p=3,$ $c=0.3,$
$h=0.1,$ $\Delta t=0.2,$ $\protect\varepsilon =3,$ $\protect\mu =1,$ $x\in %
\left[ 0,80\right] .$ }
\label{402}\vskip-1.cm
\par
\begin{center}
{\scriptsize
\begin{tabular}{cccccc}
\hline\hline
$Time$ & $I_{1}$ & $I_{2}$ & $I_{3}$ & $L_{2}$ & $L_{\infty }$ \\ \hline
0 & 2.8043580 & 2.4664883 & 0.9855618 & 0.00000000 & 0.00000000 \\
5 & 2.8043723 & 2.4665080 & 0.9855942 & 0.00183258 & 0.00177948 \\
10 & 2.8043747 & 2.4665108 & 0.9855973 & 0.00291958 & 0.00233283 \\
15 & 2.8043753 & 2.4665119 & 0.9855973 & 0.00372417 & 0.00285444 \\
20 & 2.8043758 & 2.4665135 & 0.9855973 & 0.00448357 & 0.00337609 \\
\hline\hline
\end{tabular}
}
\end{center}
\end{table}
\begin{table}[h!]
\caption{Comparisons of results for single solitary wave with $p=3,$ $c=0.3,$
$h=0.1,$ $\Delta t=0.2,$ $\protect\varepsilon =3,$ $\protect\mu =1,$ $x\in %
\left[ 0,80\right] $ at $t=20.$}
\label{403}\vskip-1.cm
\par
\begin{center}
{\scriptsize
\begin{tabular}{cccccc}
\hline\hline
$Method$ & $I_{1}$ & $I_{2}$ & $I_{3}$ & $L_{2}$ & $L_{\infty }$ \\ \hline
Our Method & 2.8043758 & 2.4665135 & 0..9855973 & 0.00448357 & 0.00337609 \\
Cubic Galerkin\cite{sbgk} & 2.8187398 & 2.4852249 & 1.0070200 & 0.01655637 &
0.01370453 \\
Quintic Collocation First Scheme\cite{sbgk6} & 2.8043570 & 2.4639086 &
0.9855602 & 0.00801470 & 0.00538237 \\
Quintic Collocation Second Scheme\cite{sbgk6} & 2.8042943 & 2.4637495 &
0.9854011 & 0.00708553 & 0.00480470 \\
Petrov-Galerkin\cite{roshan} & 2.80436 & 2.46389 & 0.98556 & 0.00484271 &
0.00370926 \\ \hline\hline
\end{tabular}
}
\end{center}
\end{table}
{\normalsize
\begin{figure}[h!]
{\normalsize {\scriptsize {\centering{\small %
\includegraphics[scale=.25]{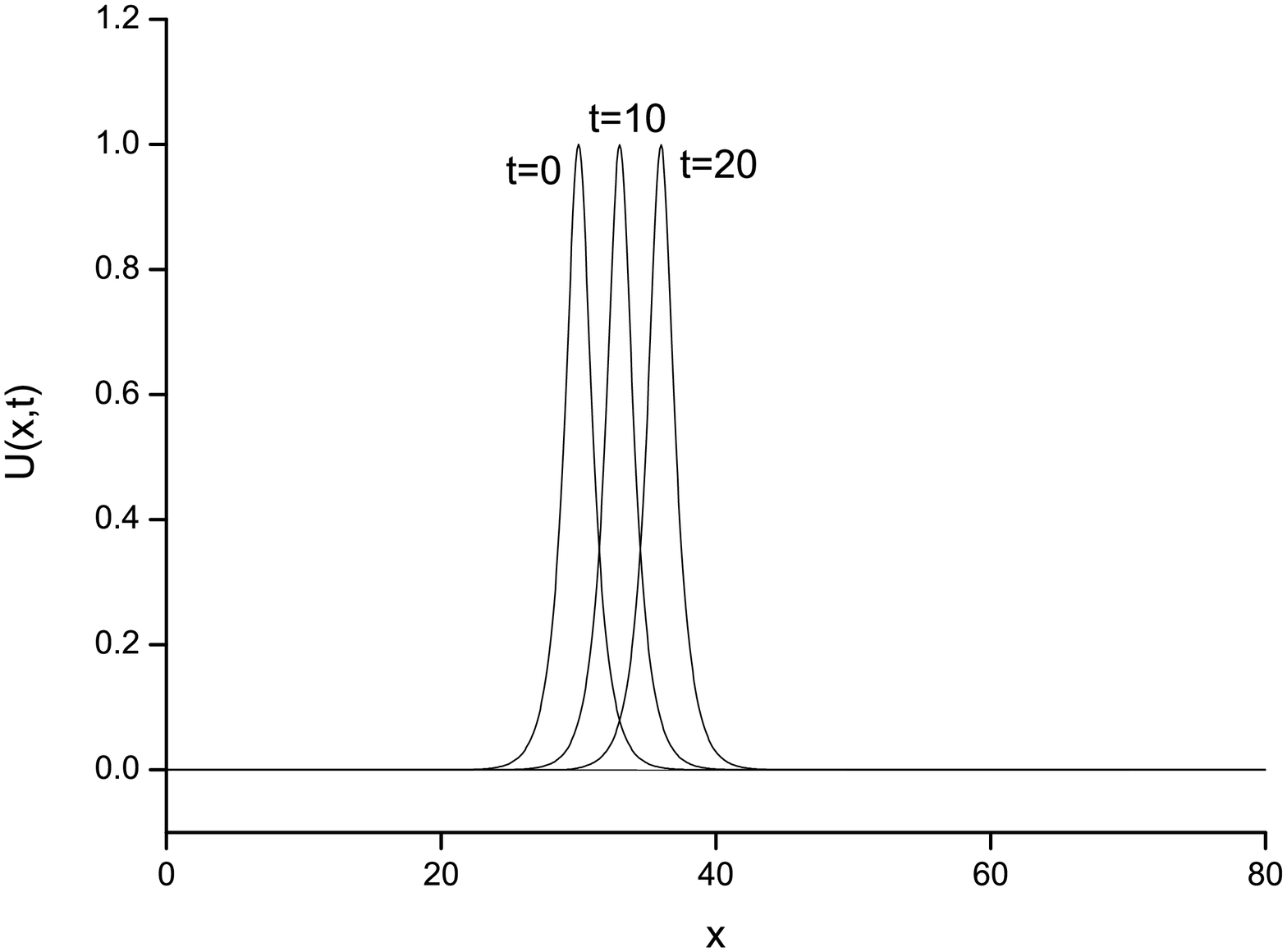}} } }  }
\caption{Motion of single solitary wave for $p=3,$ $c=0.3,$ $h=0.1,$ $\Delta
t=0.2,$ $\protect\varepsilon =3,$ $\protect\mu =1,$ $x\in \left[ 0,80\right]
$ at $t=0,10,20.$}
\label{502}
\end{figure}
\begin{figure}[h!]
{\normalsize {\scriptsize {\centering{\small %
\includegraphics[scale=.25]{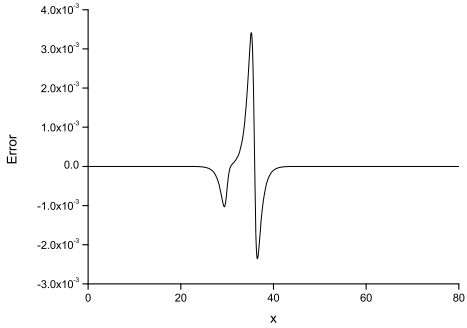}} } }  }
\caption{Error graph for $p=3,$ $c=0.3,$ $h=0.1,$ $\Delta t=0.2,$ $\protect%
\varepsilon =3,$ $\protect\mu =1,$ $x\in \left[ 0,80\right] $ at $t=20$.}
\label{503}
\end{figure}
\ \ \ \ }

For our final treatment, we put the parameters $p=4,$ $c=0.2,$ $h=0.1,$ $%
\Delta t=0.2,$ $\varepsilon =3,$ $\mu =1$, $\ x_{0}$ $=30$ over the interval
$[0,80]$ to make possible comparisons with those of earlier papers \cite%
{sbgk,sbgk6,roshan}. So solitary wave has amplitude $1.0$ and the
simulations are executed to time $t=20$ to invent the error norms $L_{2}$
and $L_{\infty }$ and the numerical invariants $I_{1},I_{2}$ and $I_{3}.$
For these values of the parameters, the conservation properties and the $%
L_{2}$-error as well as the $L_{\infty }$-error norms have been listed in
Table$(\ref{4040})$ for several values of the time level $t$. It can be
referred from Table$(\ref{4040}),$ the error norms $\mathit{L}_{2}$ and $%
\mathit{L}_{\infty }$ remain less than $1.96046\times 10^{-3}$, $%
1.33416\times 10^{-3}$ and they are still small when the time is increased
up to $t=20$ and the invariants $I_{1},I_{2}$, $I_{3}$ change from their
initial values by less than $4.07\times \ 10^{-5},$ $5.80\times \ 10^{-5}$
and $6.32\times \ 10^{-5},$ respectively, throughout the simulation. Hence
we can say our method is sensibly conservative. The comparison between the
results obtained by the current method with those in the other papers \cite%
{sbgk,sbgk6,roshan} is also documented in Table$(\ref{4050})$. It is
noticeably seen from the table that errors of the current method are
radically less than those obtained with the earlier methods \cite%
{sbgk,sbgk6,roshan}. For visual representation, the simulations of single
soliton for values $p=4,c=0.2,h=0.1,\Delta t=0.2$ at times $t=0,10$ and $20$
are illustrated in Figure$(\ref{504})$. It is understood from this figure
that the numerical scheme performs the motion of propagation of a single
solitary wave, which moves to the right at nearly unchanged speed and
conserves its amplitude and shape with increasing time. The amplitude is $%
1.00000$ at $t=0$ and located at $x=30$, while it is $0.999475$ at $t=20$
and located at $x=34$. The absolute difference in amplitudes at times $t=0$
and $t=10$ is $5.25\times 10^{-4}$ so that there is a little change between
amplitudes. Error distributions at time $t=20$ are shown graphically in
Figure$(\ref{505})$. As it is seen, the maximum errors are between $%
-1.5\times 10^{-3}$ to $1.5\times 10^{-3}$ and occur around the central
position of the solitary wave.
\begin{table}[h!]
\caption{Invariants and errors for single solitary wave with $p=4,$ $c=0.2,$
$h=0.1,$ $\Delta t=0.2,$ $\protect\varepsilon =3,$ $\protect\mu =1,$ $x\in %
\left[ 0,80\right] .$ }
\label{4040}\vskip-1.cm
\par
\begin{center}
{\scriptsize
\begin{tabular}{cccccc}
\hline\hline
$Time$ & $I_{1}$ & $I_{2}$ & $I_{3}$ & $L_{2}$ & $L_{\infty }$ \\ \hline
0 & 2.6220516 & 2.3598323 & 0.7853952 & 0.00000000 & 0.00000000 \\
5 & 2.6220846 & 2.3598808 & 0.7854675 & 0.00125061 & 0.00141788 \\
10 & 2.6220915 & 2.3598891 & 0.7854783 & 0.00178634 & 0.00147002 \\
15 & 2.6220920 & 2.3598898 & 0.7854785 & 0.00193428 & 0.00139936 \\
20 & 2.6220923 & 2.3598903 & 0.7854785 & 0.00196046 & 0.00133416 \\
\hline\hline
\end{tabular}
}
\end{center}
\end{table}
\begin{table}[h!]
\caption{Comparisons of results for single solitary wave with $p=4,$ $c=0.2,$
$h=0.1,$ $\Delta t=0.2,$ $\protect\varepsilon =3,$ $\protect\mu =1,$ $x\in %
\left[ 0,100\right] $ at $t=20.$ }
\label{4050}\vskip-1.cm
\par
\begin{center}
{\scriptsize
\begin{tabular}{cccccc}
\hline\hline
$Method$ & $I_{1}$ & $I_{2}$ & $I_{3}$ & $L_{2}$ & $L_{\infty }$ \\ \hline
Our Method & 2.6220923 & 2.3598903 & 0.7854785 & 0.00196046 & 0.00133416 \\
Cubic Galerkin\cite{sbgk} & 2.6327833 & 2.3730032 & 0.8023383 & 0.00890617 &
0.00821991 \\
Quintic Collocation First Scheme\cite{sbgk6} & 2.6220508 & 2.3561901 &
0.7853939 & 0.00421697 & 0.00297952 \\
Quintic Collocation First Scheme\cite{sbgk6} & 2.6219284 & 2.3559327 &
0.7851364 & 0.00339086 & 0.00247031 \\
Petrov-Galerkin\cite{roshan} & 2.62206 & 2.35615 & 0.78534 & 0.00230499 &
0.00188285 \\ \hline\hline
\end{tabular}
}
\end{center}
\end{table}
{\normalsize
\begin{figure}[h!]
{\normalsize {\scriptsize {\centering{\small %
\includegraphics[scale=.25]{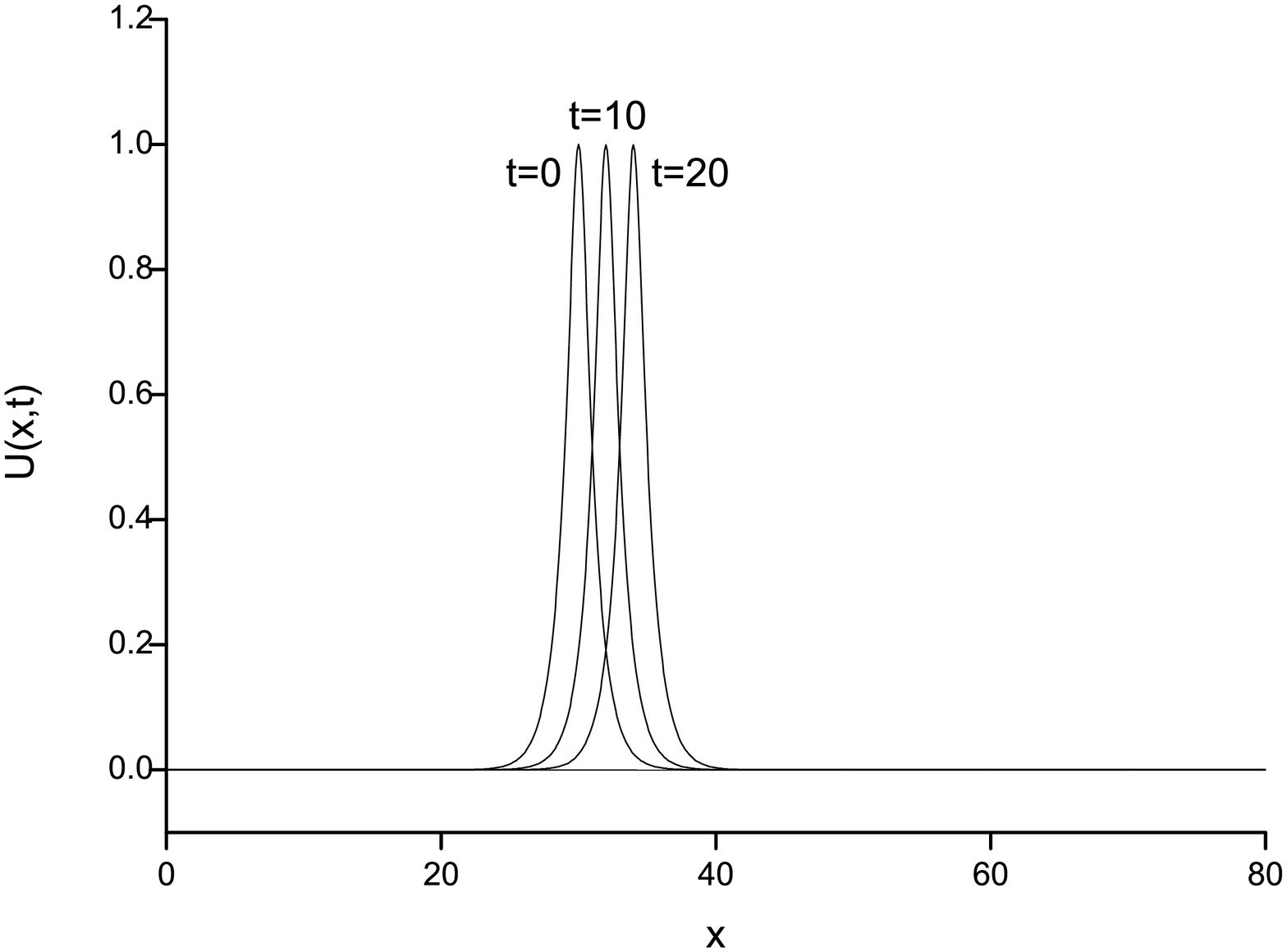}} } }  }
\caption{Motion of single solitary wave for $p=4,$ $c=0.2,$ $h=0.1,$ $\Delta
t=0.2,$ $\protect\varepsilon =3,$ $\protect\mu =1,x\in $ $[0,80]$ at $%
t=0,10,20.$}
\label{504}
\end{figure}
\begin{figure}[h!]
{\normalsize {\scriptsize {\centering{\small %
\includegraphics[scale=.25]{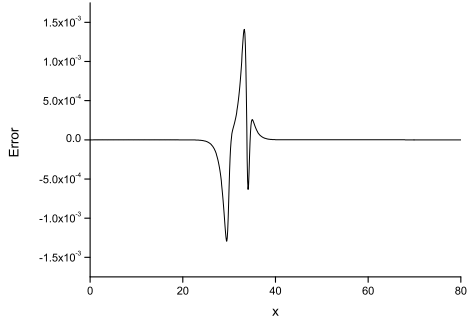}} } }  }
\caption{Error graph for $p=4,$ $c=0.2,$ $h=0.1,$ $\Delta t=0.2,$ $\protect%
\varepsilon =3,$ $\protect\mu =1$ at $t=20$.}
\label{505}
\end{figure}
}

\subsection{Interaction of two solitary waves}

Our second test problem pertains to the interaction of two solitary wave
solutions of GEW equation having different amplitudes and traveling in the
same direction. We tackle GEW equation with initial conditions given by the
linear sum of two well separated solitary waves of various amplitudes as
follows

\begin{equation}
U(x,0)=\sum_{j=1}^{2}\sqrt[p]{\frac{c_{j}(p+1)(p+2)}{2\varepsilon }\sec
h^{2}[\frac{p}{2\sqrt{\mu }}(x-x_{j})]},  \label{ini. for two solit}
\end{equation}%
where $c_{j}$ and $x_{j}$, $\ j=1,2$ are arbitrary constants. For the
computational work, two sets of parameters are considered by taking
different values of $p,c_{i}$ and the same values of $h=0.1$, $\Delta
t=0.025,$ $\varepsilon =3,$ $\mu =1$ over the interval $0\leq x\leq 80.$ We
firstly take $p=3,$ $c_{1}=0.3,$ $c_{2}=0.0375.$ So the amplitudes of the
two solitary waves are in the ratio $2:1.$ Calculations are done up to $%
t=100 $. The three invariants in this case are tabulated in Table$(\ref{4051}%
)$ . It is clear that the quantities are satisfactorily constant and very
closed with the methods \cite{sbgk,sbgk6,roshan} during the computer run.
Fig. $(\ref{40510})$ illustrates the behavior of the interaction of two
positive solitary waves. At $t=100$, the magnitude of the smaller wave is $%
0.510619$ on reaching position $x=31.8$, and of the larger wave $0.999364$
having the position $x=46.7$, so that the difference in amplitudes is $%
0.010619$ for the smaller wave and $0.000636$ for the larger wave. For the
second case, we have studied the interaction of two solitary waves with the
parameters \ $p=4,c_{1}=0.2,$ $c_{2}=1/80.$ So the amplitudes of the two
solitary waves are in the ratio $2:1$. \ For this case the experiment is run
until time $t=120$. The three invariants in this case are recorded in Table$(%
\ref{40530}).$ The results in this table indicate that the numerical values
of the invariants are good agreement with those of methods\cite%
{sbgk,sbgk6,roshan} during the computer run. Fig.$(\ref{4052})$ shows the
development of the solitary wave interaction$.$
\begin{table}[h!]
\caption{Invariants for interaction of two solitary waves with $p=3.$ }
\label{4051}\vskip-1.cm
\par
\begin{center}
{\scriptsize
\begin{tabular}{ccccccc}
\hline\hline
&  &  &  &  &  &  \\ \hline
& $t$ & $0$ & $30$ & $60$ & $90$ & $100$ \\ \hline
& Our Method & 4.20653 & 4.20657 & 4.20622 & 4.20502 & 4.20517 \\
$I_{1}$ & \cite{sbgk} & 4.20653 & 4.20653 & 4.20616 & 4.20490 & 4.20503 \\
& \cite{sbgk6} First & 4.20653 & 4.20653 & 4.20653 & 4.20653 & 4.20653 \\
& \cite{sbgk6} Second & 4.20653 & 4.20653 & 4.20653 & 4.20653 & 4.20653 \\
& \cite{roshan} & 4.20655 & 4.20655 & 4.20655 & 4.20655 & 4.20655 \\
& Our Method & 3.08311 & 3.08318 & 3.08309 & 3.08220 & 3.08251 \\
$I_{2}$ & \cite{sbgk} & 3.07987 & 3.07991 & 3.07947 & 3.07777 & 3.07797 \\
& \cite{sbgk6} First & 3.07988 & 3.07988 & 3.07988 & 3.07988 & 3.07988 \\
& \cite{sbgk6} Second & 3.07988 & 3.07988 & 3.07988 & 3.07988 & 3.07988 \\
& \cite{roshan} & 3.97977 & 3.07980 & 3.07987 & 3.07974 & 3.07972 \\
& Our Method & 1.01636 & 1.01644 & 1.01664 & 1.01632 & 1.01634 \\
$I_{3}$ & \cite{sbgk} & 1.01636 & 1.01638 & 1.01654 & 1.01616 & 1.01616 \\
& \cite{sbgk6} First & 1.01636 & 1.01636 & 1.01636 & 1.01636 & 1.01636 \\
& \cite{sbgk6} Second & 1.01636 & 1.01636 & 1.01636 & 1.01636 & 1.01636 \\
& \cite{roshan} & 1.01634 & 1.01634 & 1.01634 & 1.01633 & 1.01634 \\
\hline\hline
\end{tabular}
}
\end{center}
\end{table}
\begin{table}[h!]
\caption{Invariants for interaction of two solitary waves with $p=4.$ }
\label{40530}\vskip-1.cm
\par
\begin{center}
{\scriptsize
\begin{tabular}{ccccccc}
\hline\hline
&  &  &  &  &  &  \\ \hline
& $t$ & $0$ & $30$ & $60$ & $90$ & 120 \\ \hline
& Our Method & 3.93307 & 3.93311 & 3.93393 & 3.93229 & 3.93037 \\
$I_{1}$ & \cite{sbgk} & 3.93307 & 3.93309 & 3.93388 & 3.93222 & 3.93026 \\
& \cite{sbgk6} First & 3.93307 & 3.93307 & 3.93307 & 3.93307 & 3.93307 \\
& \cite{sbgk6} Second & 3.93307 & 3.93307 & 3.93307 & 3.93307 & 3.93307 \\
& \cite{roshan} & 3.93309 & 3.93309 & 3.93309 & 3.93309 & 3.93308 \\
& Our Method & 2.94979 & 2.94985 & 2.95122 & 2.94939 & 2.94801 \\
$I_{2}$ & \cite{sbgk} & 2.94521 & 2.94527 & 2.94703 & 2.94436 & 2.94212 \\
& \cite{sbgk6} First & 2.94524 & 2.94524 & 2.94524 & 2.94524 & 2.94524 \\
& \cite{sbgk6} Second & 2.94524 & 2.94523 & 2.94523 & 2.94523 & 2.94523 \\
& \cite{roshan} & 2.94512 & 2.94510 & 2.94505 & 2.94520 & 2.94511 \\
& Our Method & 0.79766 & 0.79775 & 0.79952 & 0.79824 & 0.79811 \\
$I_{3}$ & \cite{sbgk} & 0.79766 & 0.79770 & 0.79942 & 0.79812 & 0.79794 \\
& \cite{sbgk6} First & 0.79766 & 0.79766 & 0.79766 & 0.79766 & 0.79766 \\
& \cite{sbgk6} Second & 0.79766 & 0.79766 & 0.79766 & 0.79766 & 0.79766 \\
& \cite{roshan} & 0.79761 & 0.79761 & 0.79762 & 0.79761 & 0.79761 \\
\hline\hline
\end{tabular}
}
\end{center}
\end{table}
\begin{figure}[h!]
\centering{\small \includegraphics[scale=.20]{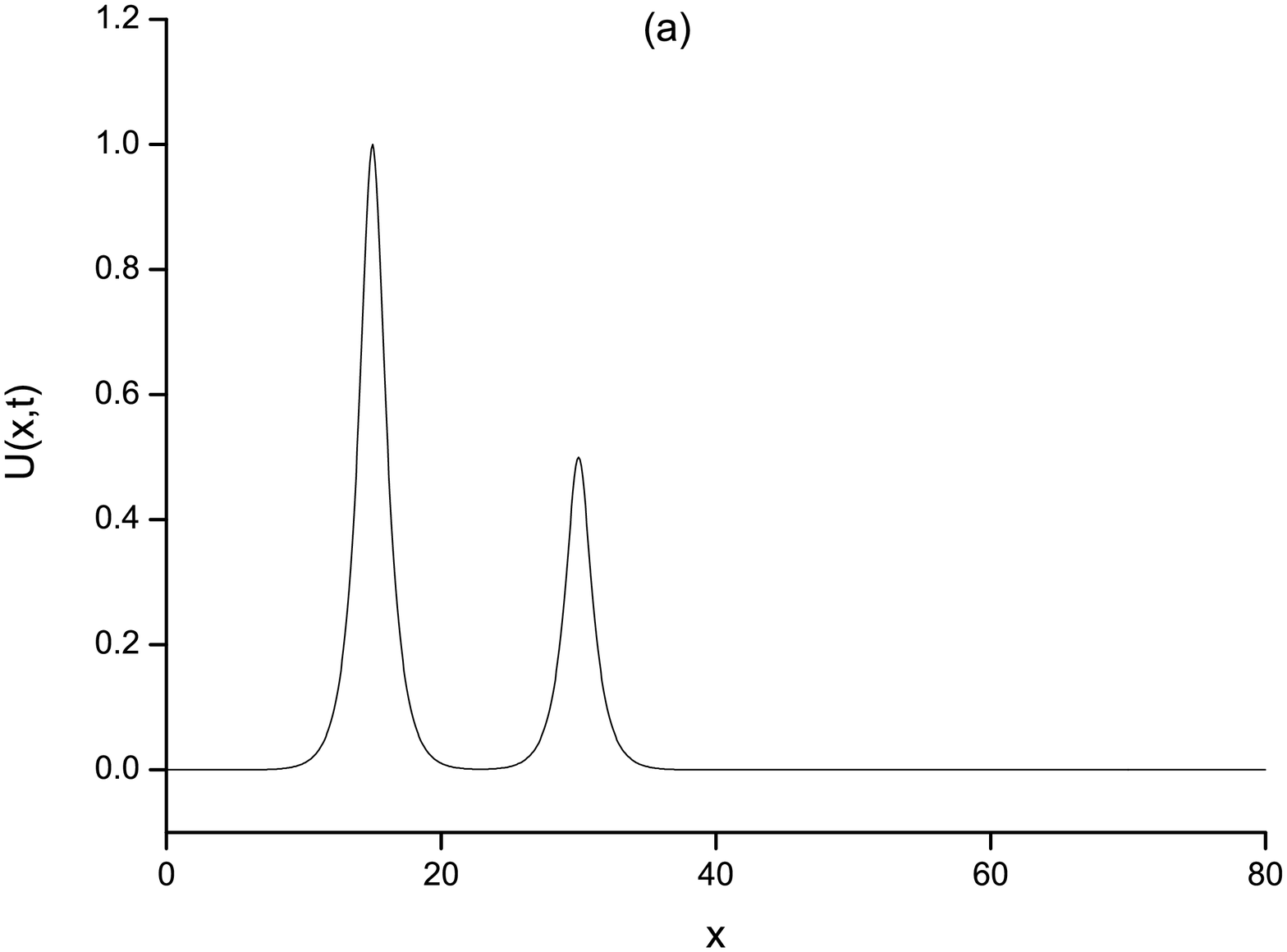} %
\includegraphics[scale=.20]{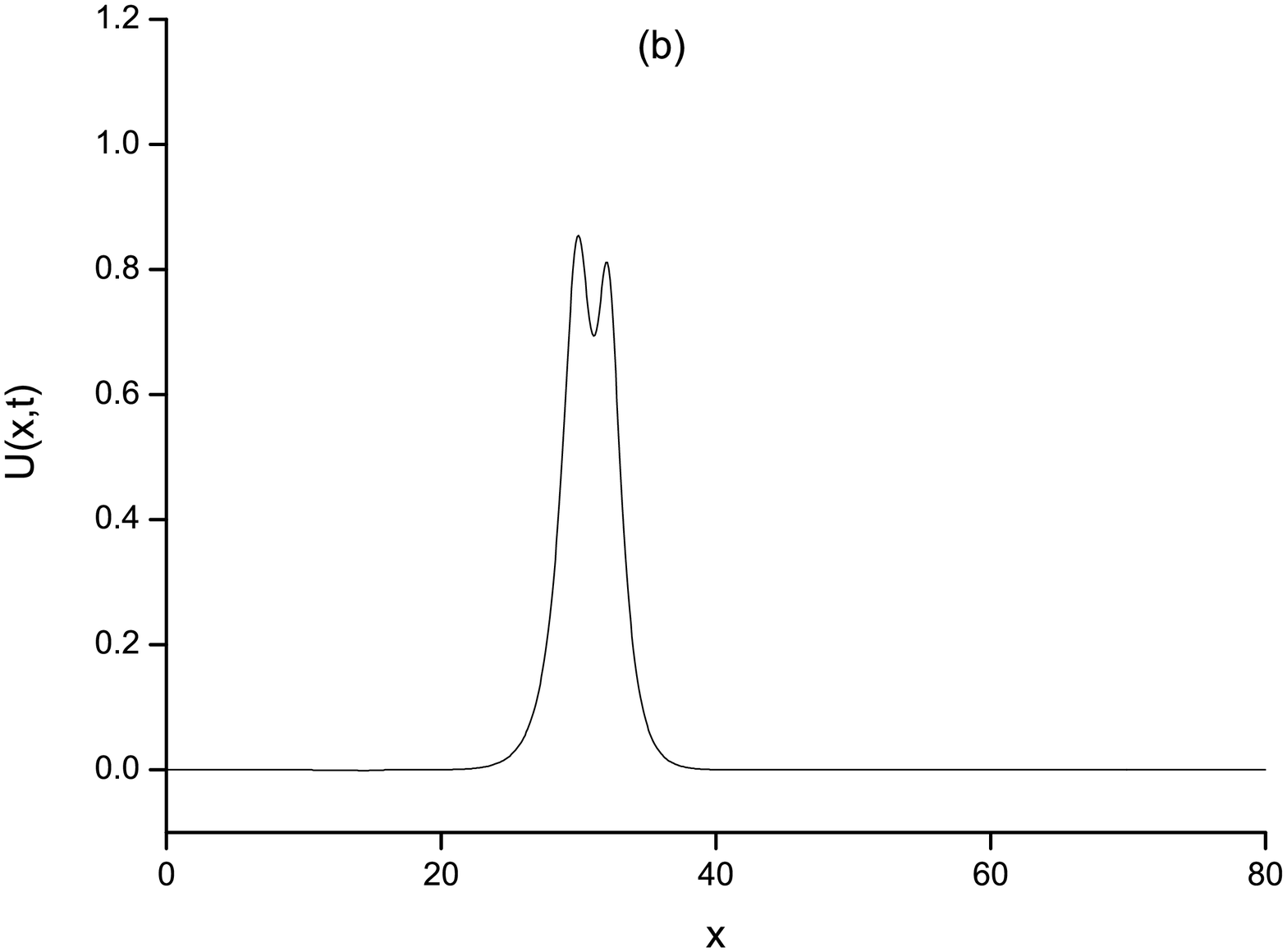}}
\par
{\small \includegraphics[scale=.20]{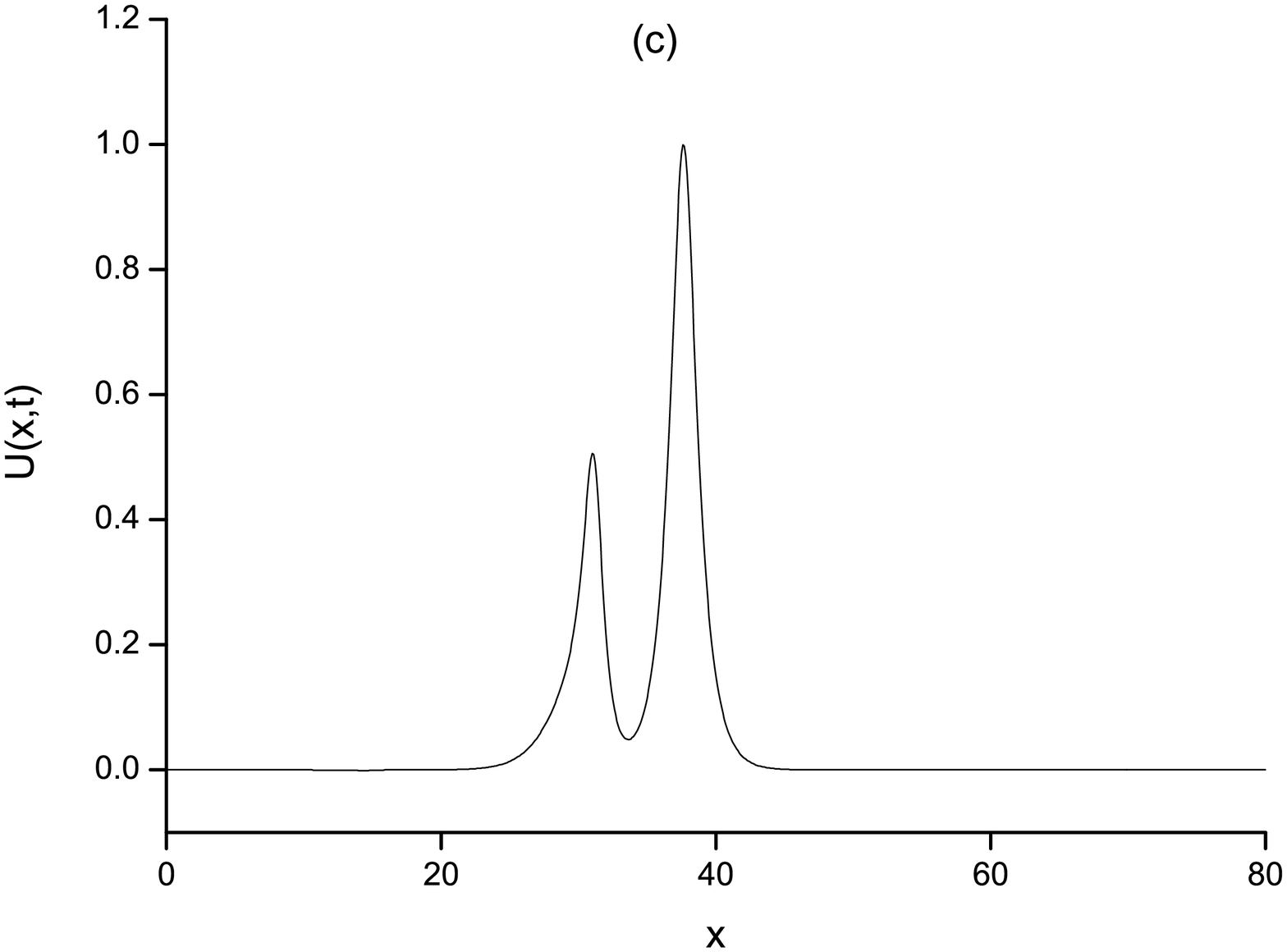} %
\includegraphics[scale=.20]{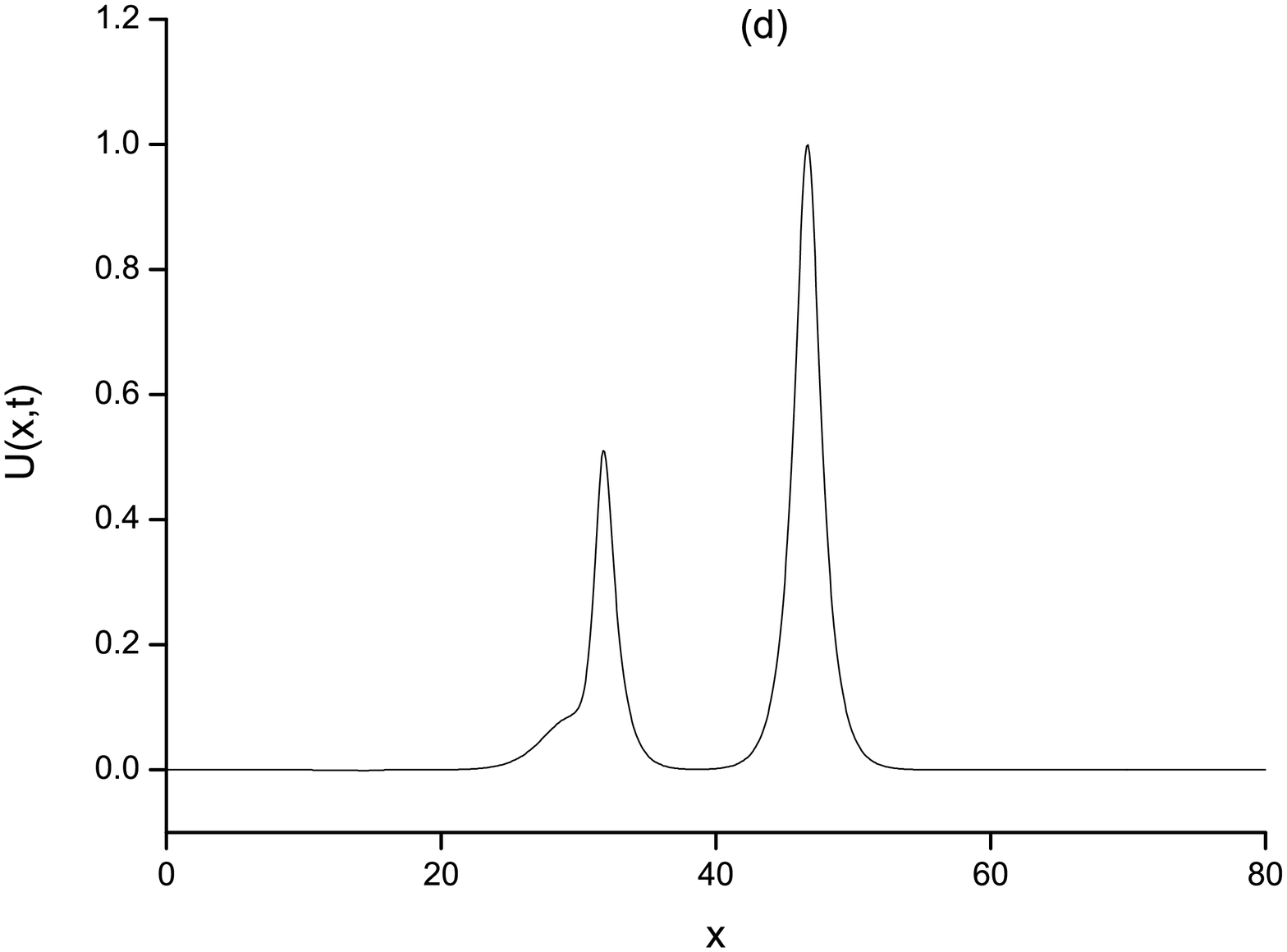}}
\caption{Interaction of two solitary waves at $p=3;$ $(a)t=0,$ $(b)t=50,$ $%
(c)t=70,$ $(d)t=100.$}
\label{40510}
\end{figure}
\begin{figure}[h!]
\centering{\small \includegraphics[scale=.20]{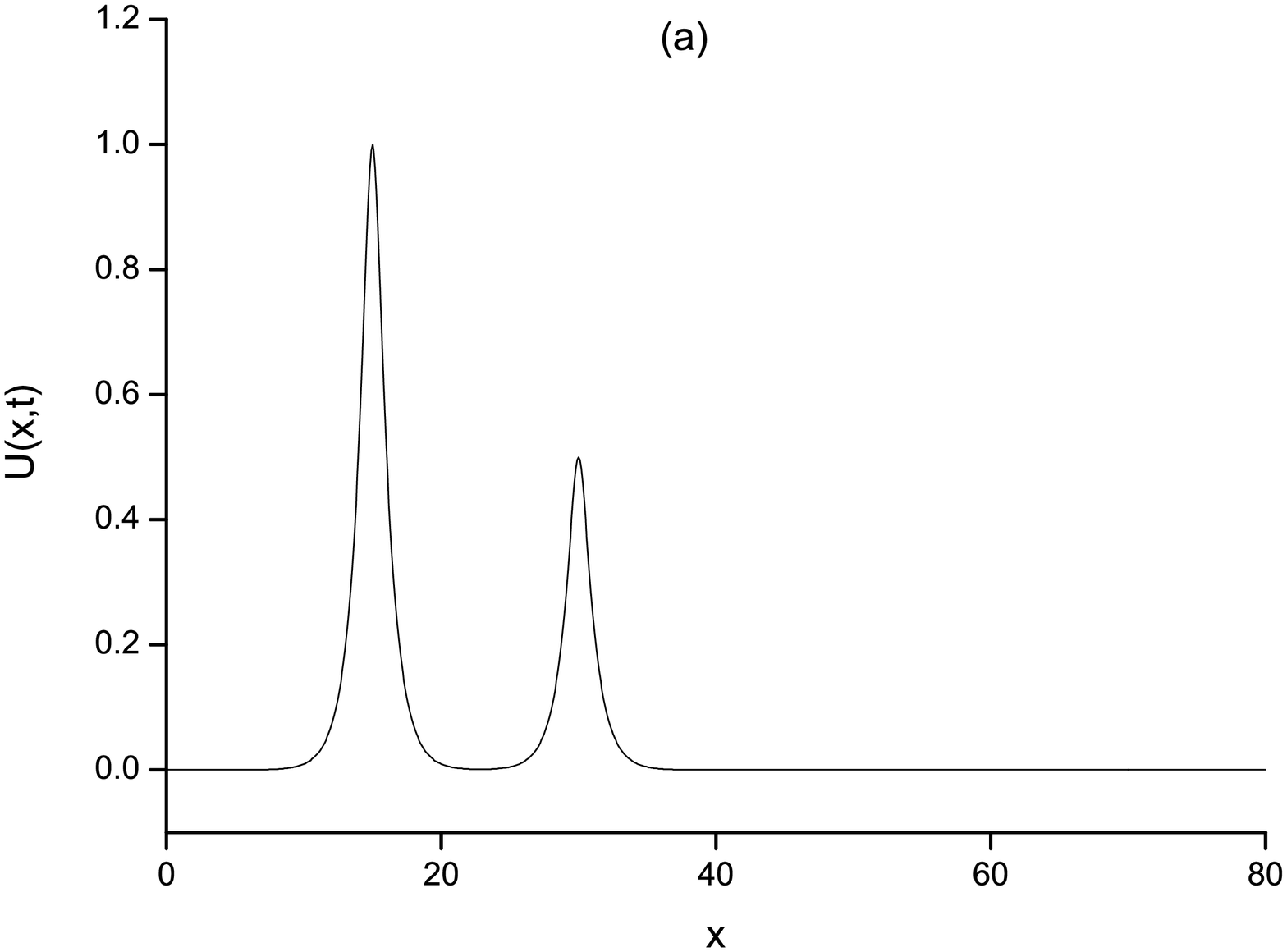} %
\includegraphics[scale=.20]{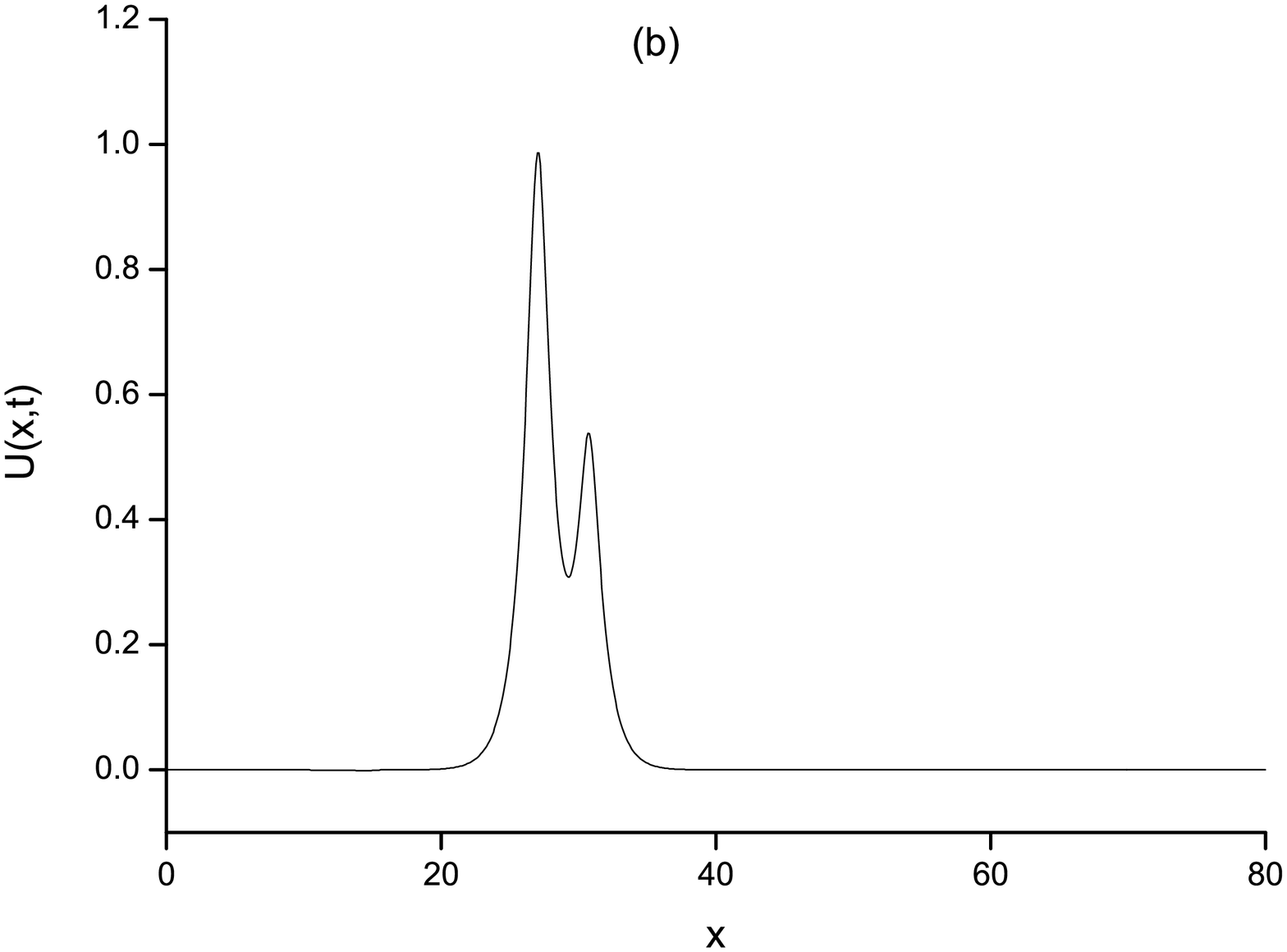}}
\par
{\small \includegraphics[scale=.20]{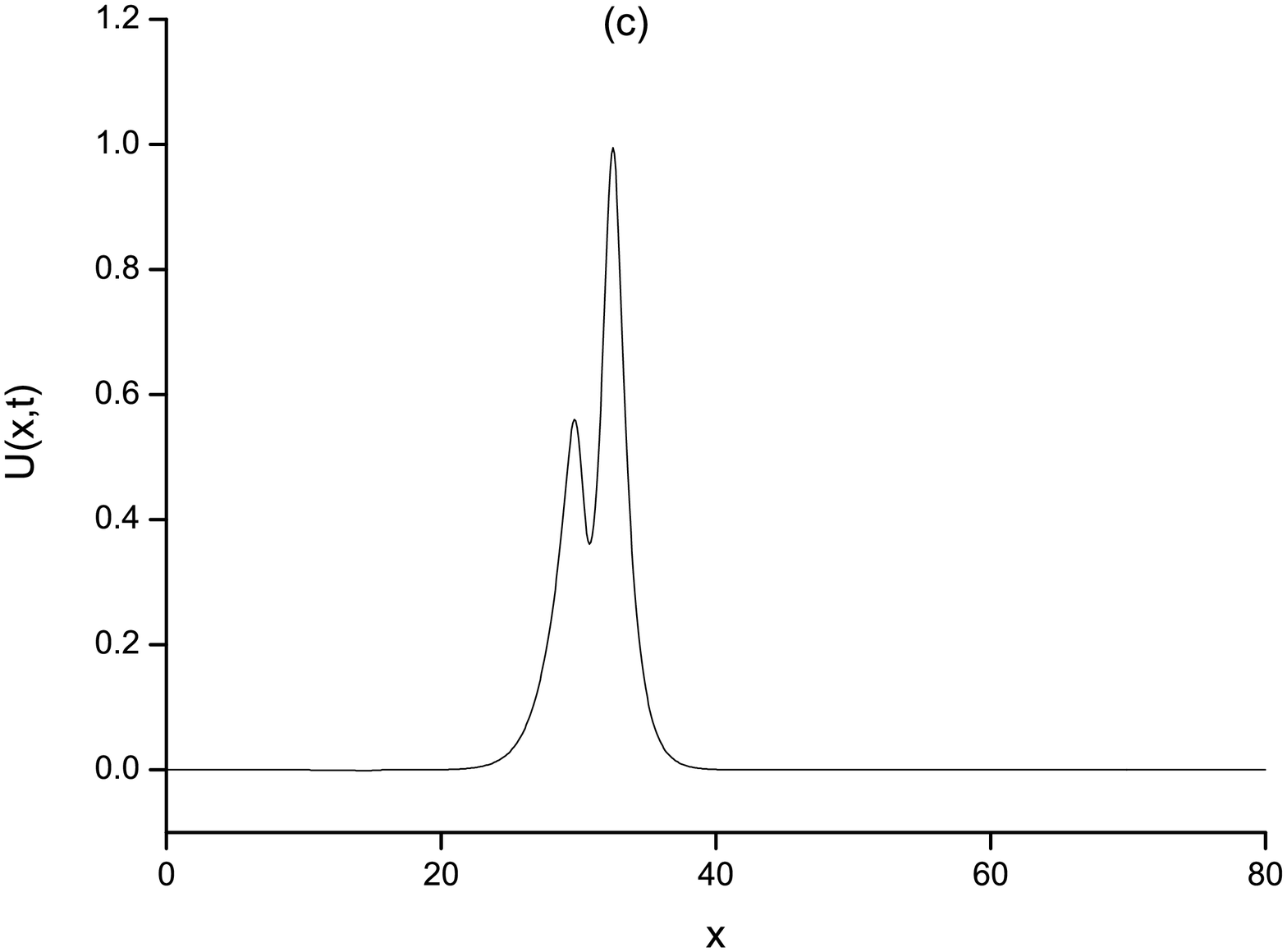} %
\includegraphics[scale=.20]{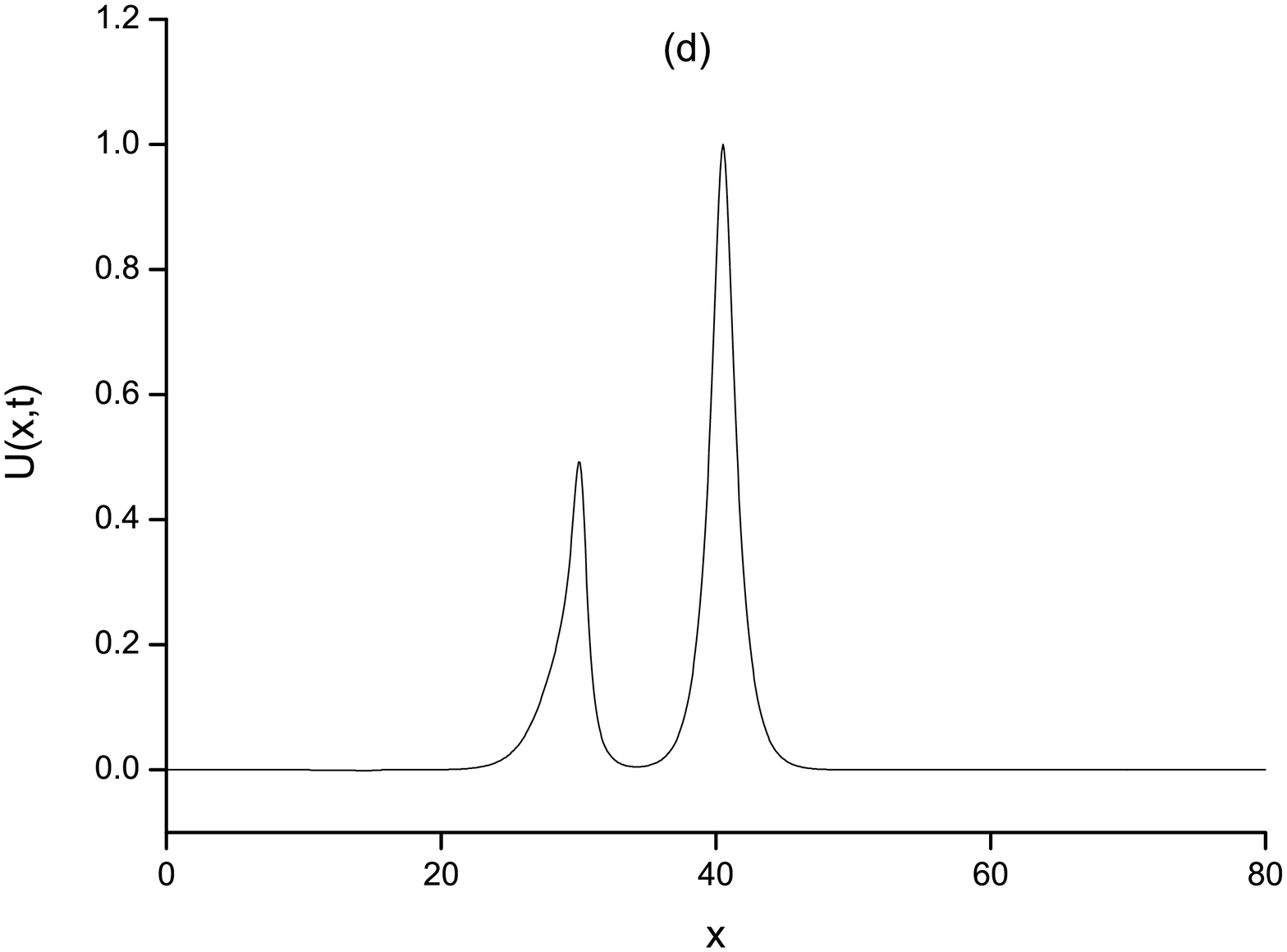}}
\caption{Interaction of two solitary waves at $p=4;$ $(a)t=0,$ $(b)t=60,$ $%
(c)t=80,$ $(d)t=120.$}
\label{4052}
\end{figure}

\subsection{Evolution of solitons}

Finally, another attracting initial value problem for the GEW equation is
evolution of the solitons that is used as the Gaussian initial condition in
solitary waves given by

\begin{equation}
U(x,0)=\exp (-x^{2}).
\end{equation}%
Since the behavior of the solution depends on values of $%
\mu
$, we choose different values of $%
\mu
=0.1$ and $%
\mu
=0.05$ for $p=2,3,4$. The numerical computations are done up to $t=12$.
Calculated numerical invariants at different values of $t$ are documented in
Table$(\ref{4053}).$ From this table, we can easily see that as the value of
$\mu $ increases, the variations of the invariants become smaller and it is
seen that calculated invariant values are satisfactorily constant. The
development of the evolution of solitons is presented in Figs.$(\ref{4055}),(%
\ref{4056})$ and $(\ref{4057})$. It is clearly seen in these figures that
when the value of $\mu $ decreases, the number of the stable solitary wave
increases.
\begin{table}[h!]
\caption{Maxwellian initial condition for different values of $\protect\mu .$%
}
\label{4053}\vskip-1.cm
\par
\begin{center}
{\scriptsize
\begin{tabular}{ccccccccccc}
\hline\hline
$\mu $ & $t$ & \multicolumn{3}{c}{p=2} & \multicolumn{3}{c}{p=3} &
\multicolumn{3}{c}{p=4} \\ \hline
&  & $I_{1}$ & $I_{2}$ & $I_{3}$ & $I_{1}$ & $I_{2}$ & $I_{3}$ & $I_{1}$ & $%
I_{2}$ & $I_{3}$ \\ \hline
& 0 & 1.7724537 & 1.3792767 & 0.8862269 & 1.7724537 & 1.3792767 & 0.7926655
& 1.7724537 & 1.3792767 & 0.7236013 \\
0.1 & 4 & 1.7724537 & 1.5760586 & 0.8862269 & 1.7724537 & 1.6168691 &
0.7926655 & 1.7724537 & 1.6360543 & 0.7236013 \\
& 8 & 1.7724537 & 1.5838481 & 0.8862269 & 1.7724537 & 1.6245008 & 0.7926655
& 1.7724537 & 1.6481131 & 0.7236013 \\
& 12 & 1.7724537 & 1.5920722 & 0.8862269 & 1.7724537 & 1.6325922 & 0.7926655
& 1.7724537 & 1.6531844 & 0.7236013 \\
\cite{sbgk6} & 12 & 1.7724 & 1.3786 & 0.8862 & 1.7724 & 1.3786 & 0.7928 &
1.7725 & 1.3786 & 0.7243 \\
\cite{roshan} & 12 & 1.7724 & 1.3785 & 0.8861 & 1.7724 & 1.3787 & 0.7926 &
1.7734 & 1.3836 & 0.7224 \\
& 0 & 1.7724537 & 1.3162954 & 0.8862269 & 1.7724537 & 1.3162954 & 0.7926655
& 1.7724537 & 1.3162954 & 0.7236013 \\
& 4 & 1.7724537 & 1.5406812 & 0.8862269 & 1.7724537 & 1.5766908 & 0.7926655
& 1.7724537 & 1.6243519 & 0.7236013 \\
0.05 & 8 & 1.7724537 & 1.6342604 & 0.8862269 & 1.7724537 & 1.6367952 &
0.7926655 & 1.7724537 & 1.6554614 & 0.7236013 \\
& 12 & 1.7724537 & 1.6835979 & 0.8862269 & 1.7724537 & 1.6372439 & 0.7926655
& 1.7724537 & 1.7079133 & 0.7236013 \\
\cite{sbgk6} & 12 & 1.7724 & 1.3159 & 0.8864 & 1.7725 & 1.3160 & 0.7940 &
1.7735 & 1.3188 & 0.7345 \\
\cite{roshan} & 12 & 1.7724 & 1.3160 & 0.8861 & 1.7724 & 1.3156 & 0.7922 &
1.7724 & 1.3177 & 0.7245 \\ \hline\hline
\end{tabular}
}
\end{center}
\end{table}
\begin{figure}[h!]
\centering{\small \includegraphics[scale=.20]{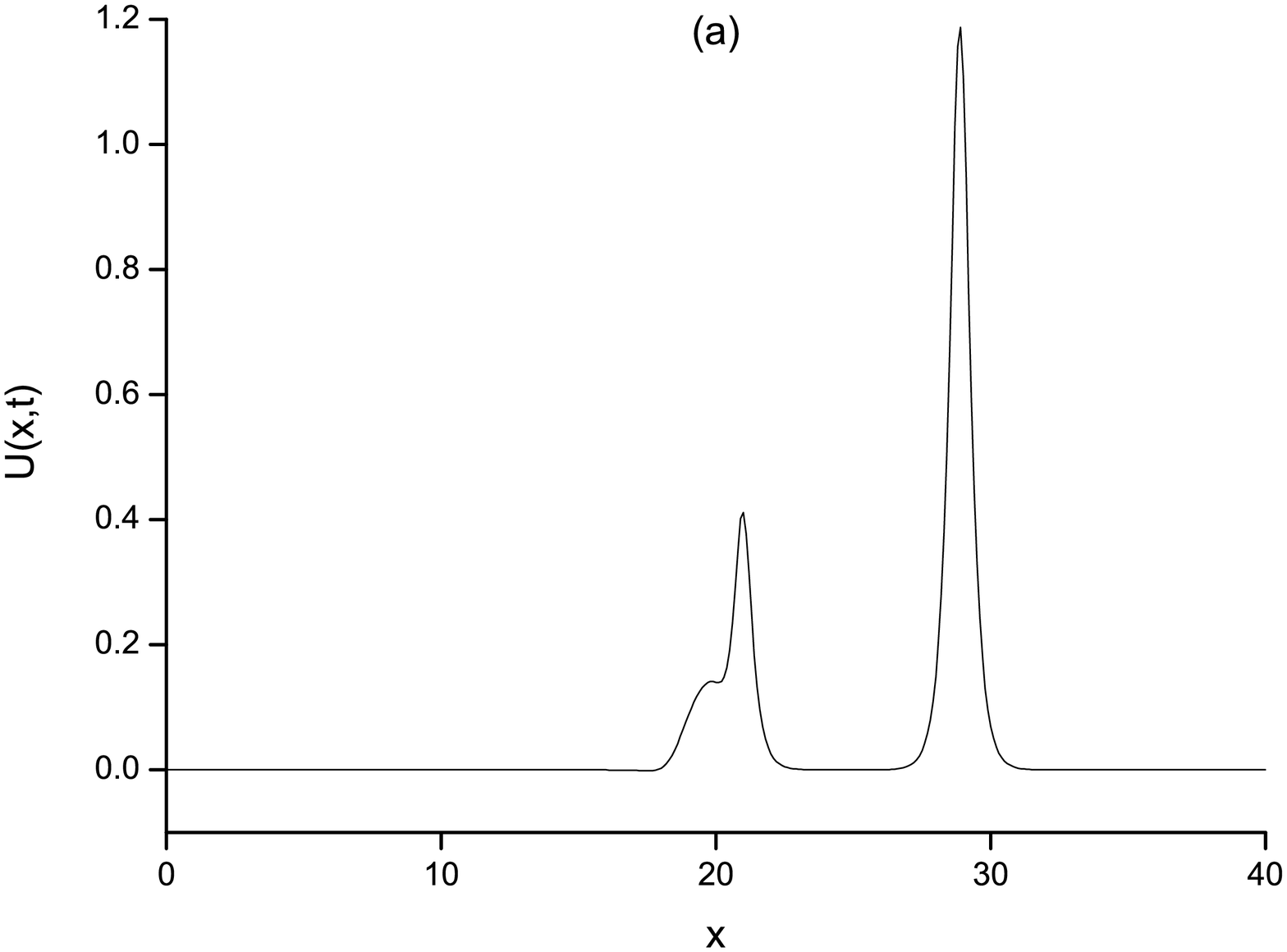} %
\includegraphics[scale=.20]{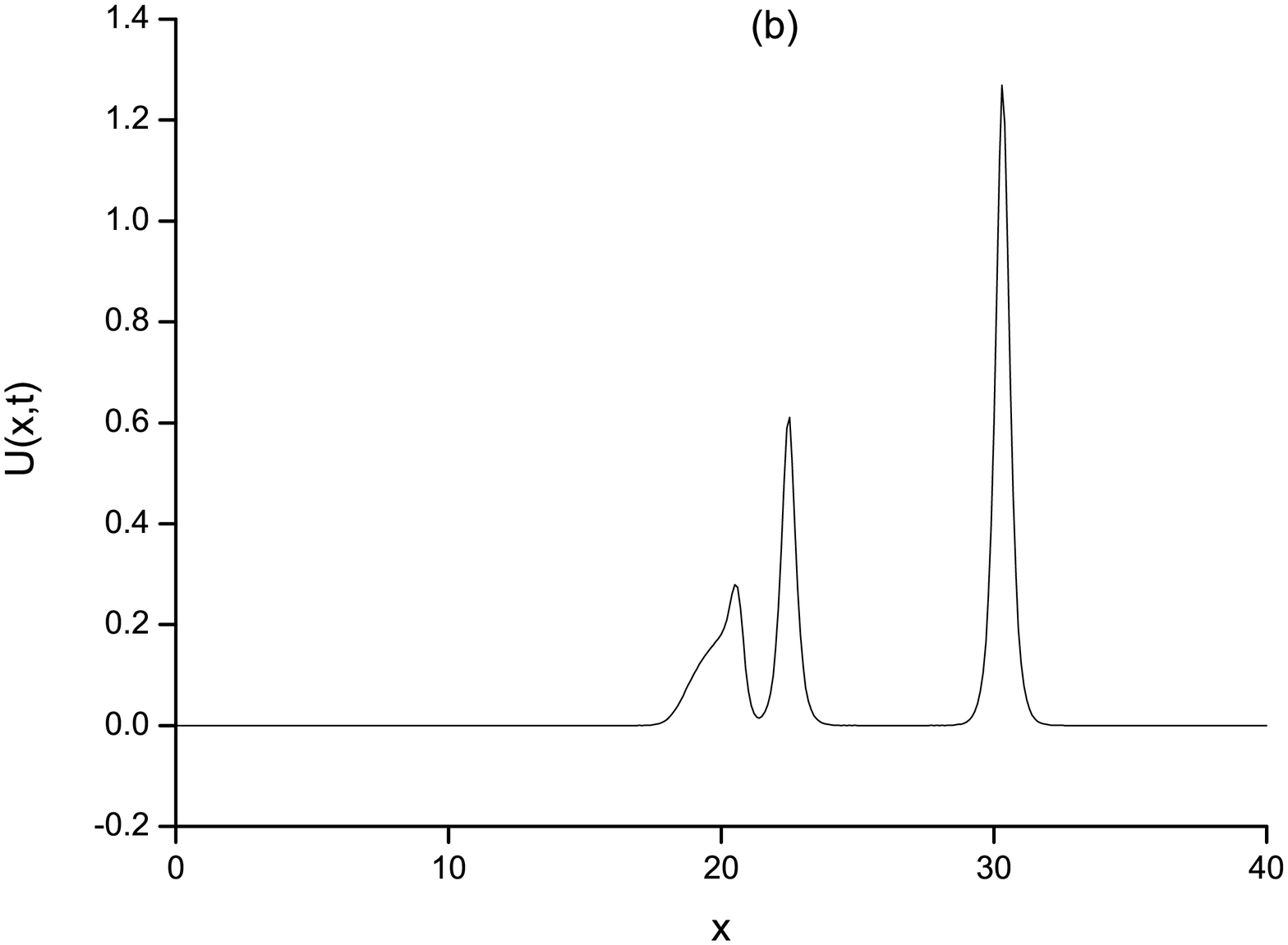}}
\caption{Maxwellian initial condition $p=2,$ $a)$ $\protect\mu =0.1,$ $b)$ $%
\protect\mu =0.05$ at $t=12.$ }
\label{4055}
\end{figure}
\begin{figure}[h!]
\centering{\small \includegraphics[scale=.20]{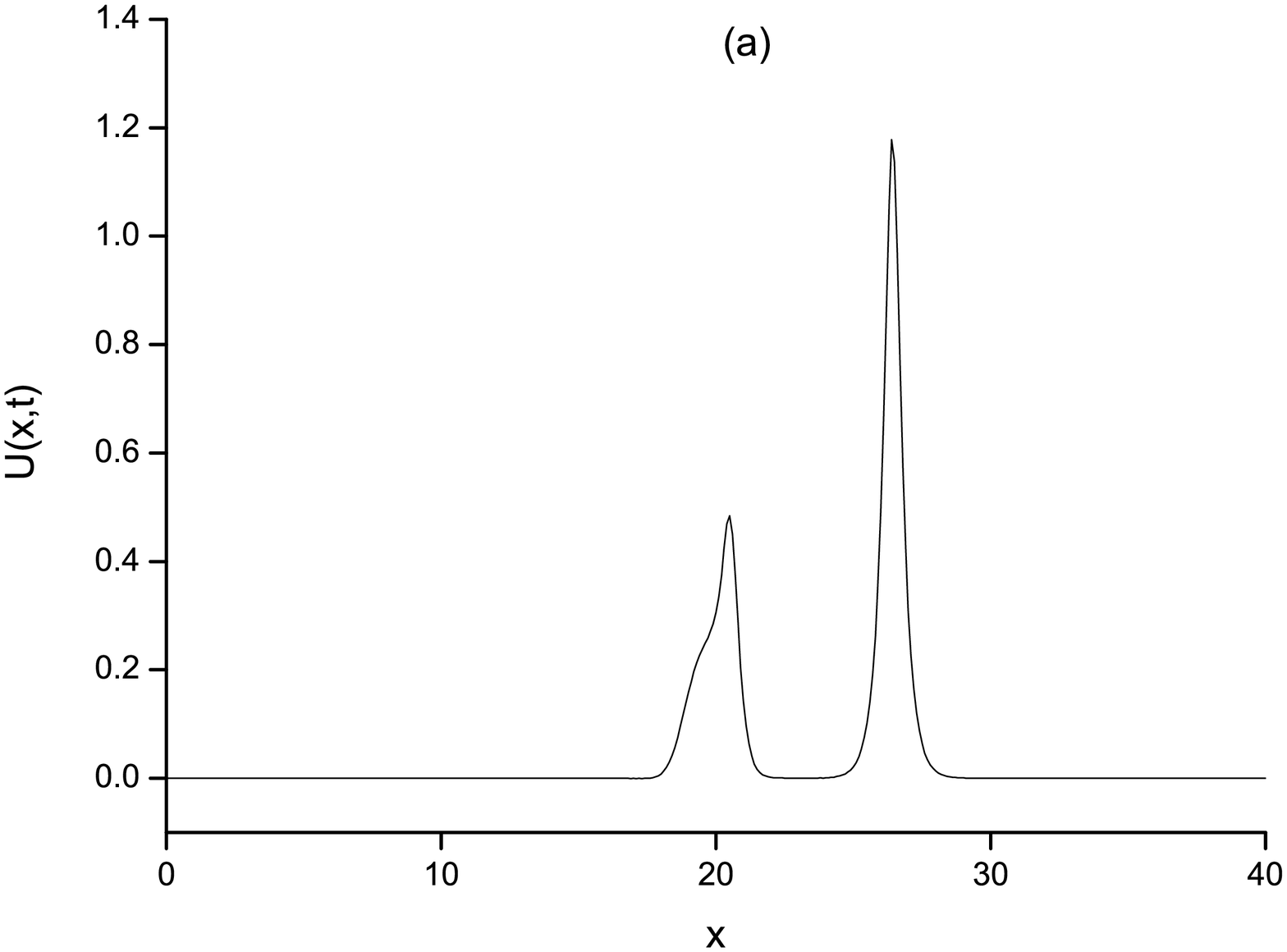} %
\includegraphics[scale=.20]{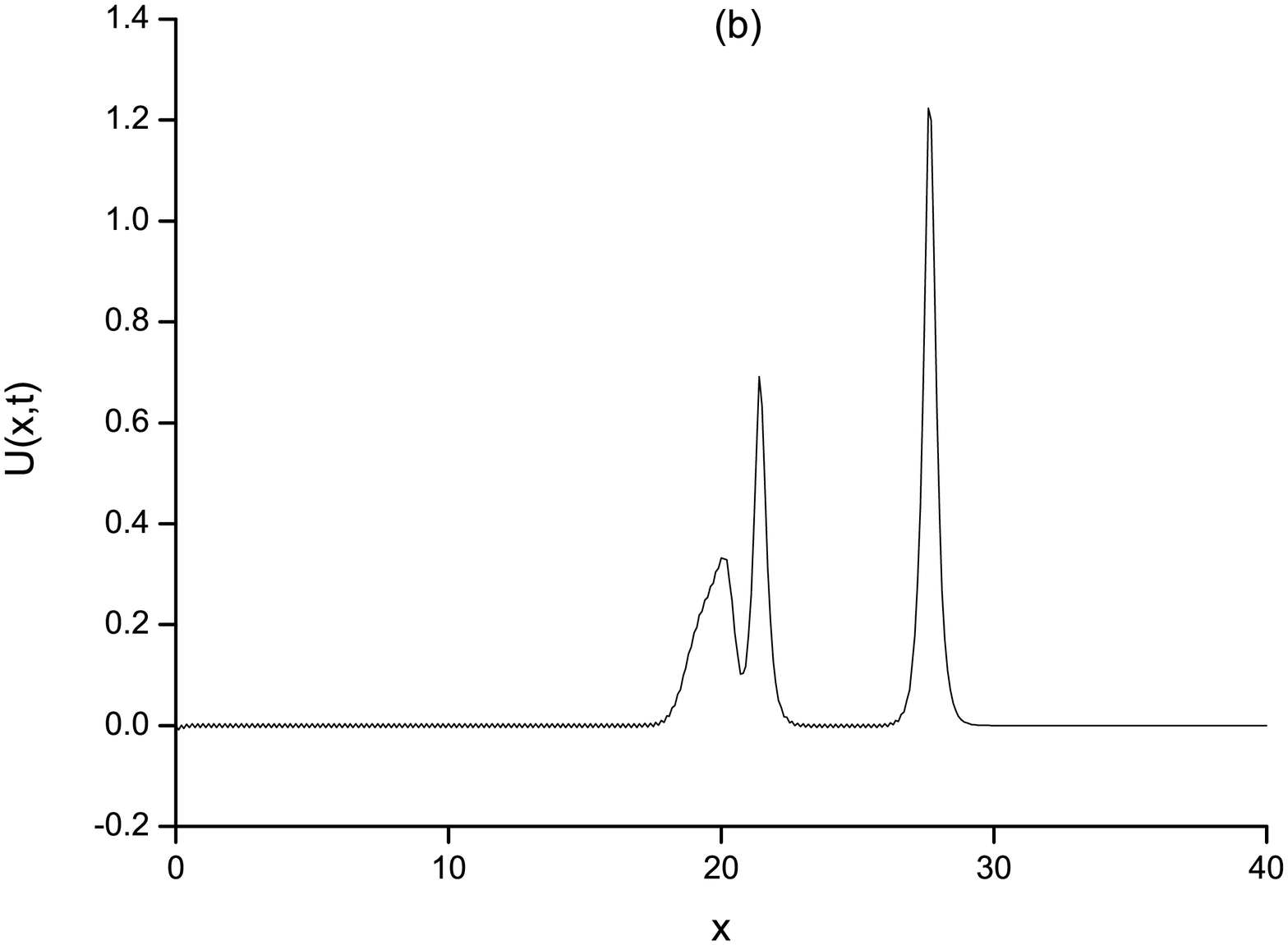}}
\caption{Maxwellian initial condition $p=3,$ $a)$ $\protect\mu =0.1,$ $b)$ $%
\protect\mu =0.05$ at $t=12.$ }
\label{4056}
\end{figure}
\begin{figure}[h!]
\centering{\small \includegraphics[scale=.20]{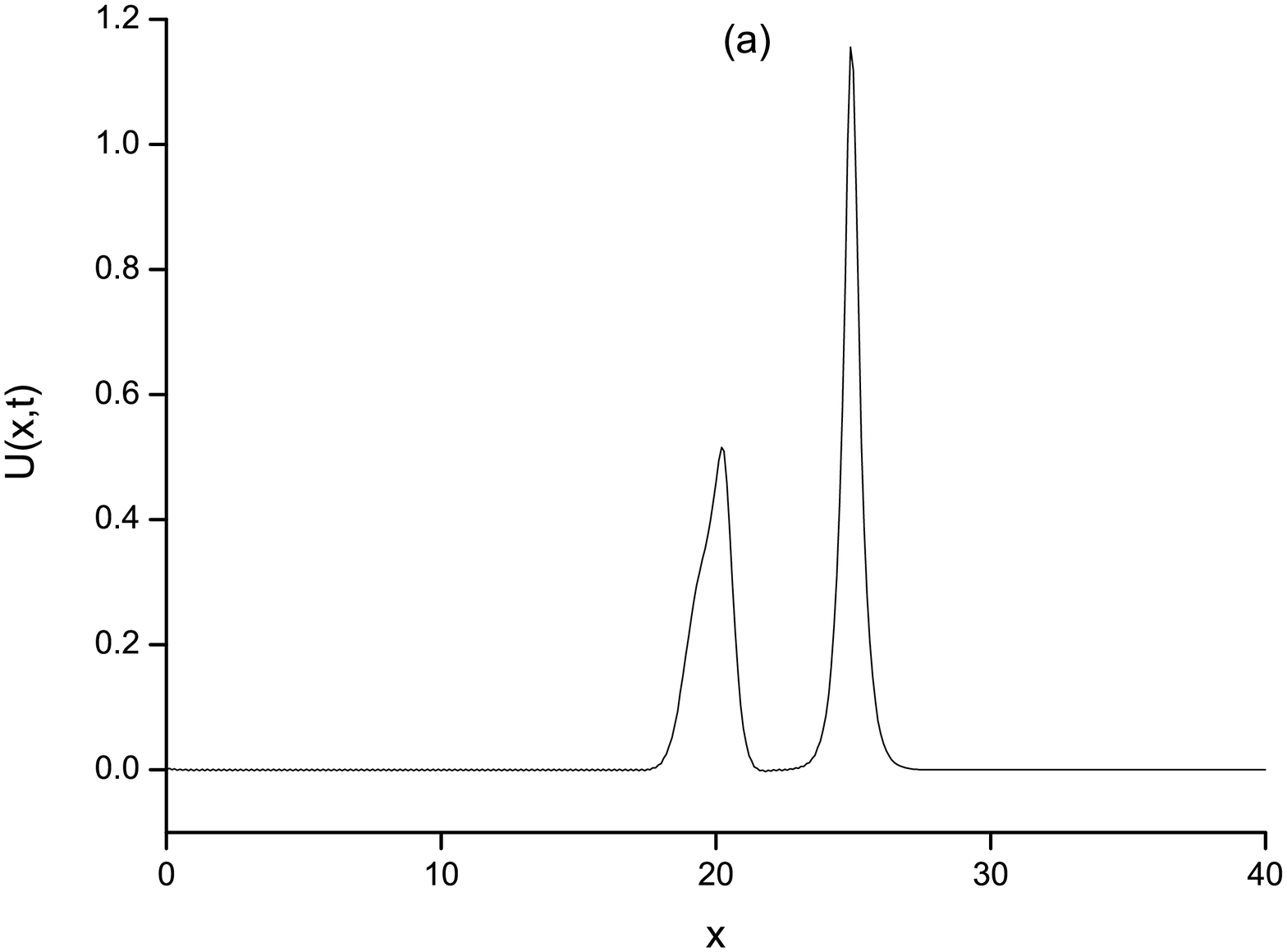} %
\includegraphics[scale=.20]{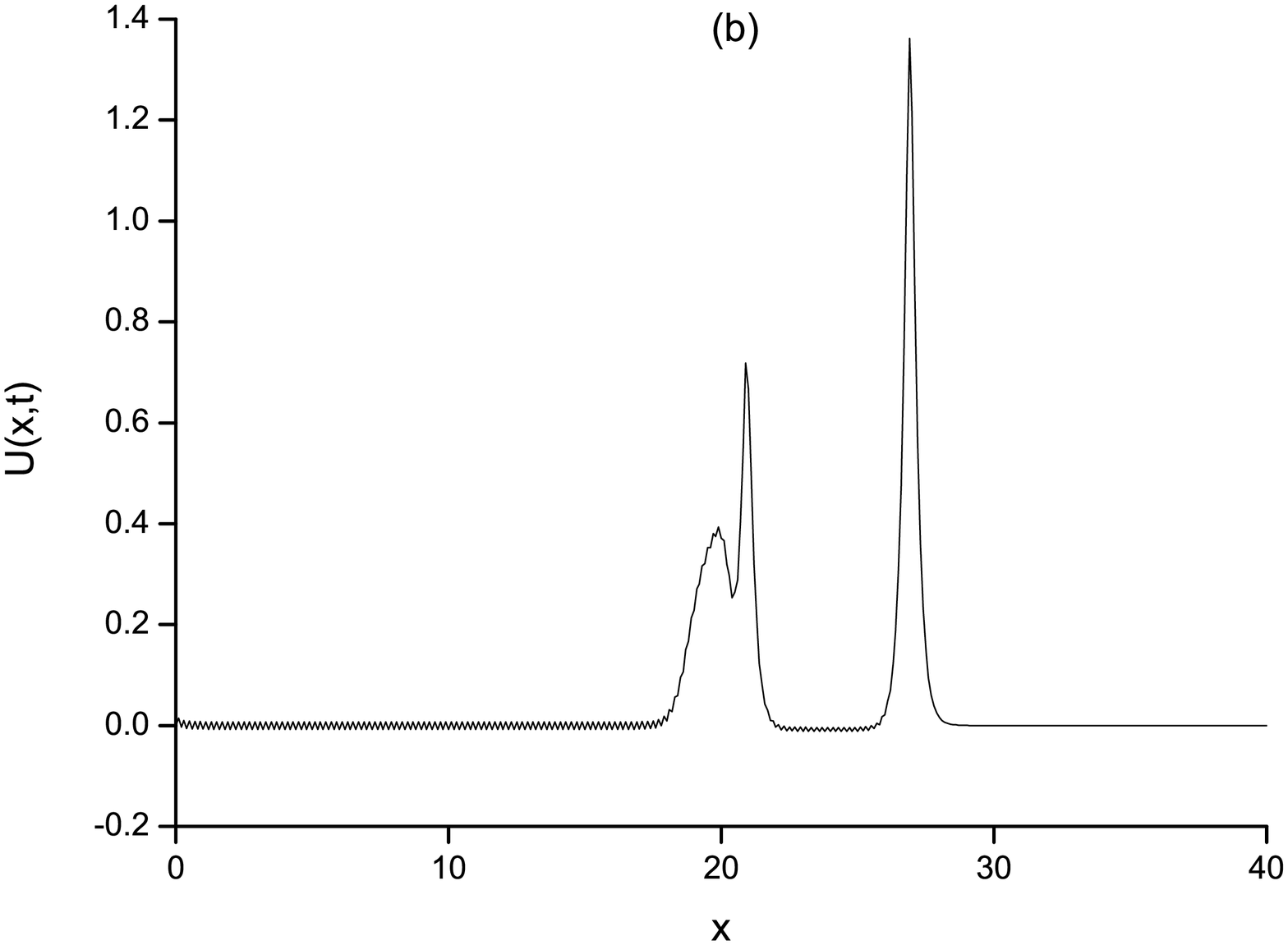}}
\caption{Maxwellian initial condition $p=4,$ $a)$ $\protect\mu =0.1,$ $b)$ $%
\protect\mu =0.05$ at $t=12.$}
\label{4057}
\end{figure}

\section{Concluding remarks}

. Solitary-wave solutions of the GEW equation by using Petrov-Galerkin
method based on linear B-spline weight functions and quadratic B-spline
trial functions, have been successfully obtained. \newline
. Existence and uniqueness of solutions of the weak of the given problem as
well as the proof of convergence has been proposed. \newline
. Solutions of a semi-discrete finite element formulation of the equation
and the theoretical bound of the error in the semi-discrete scheme are
demonstrated. \newline
. The theoretical upper bound of the error in such a full discrete
approximation at $t=t^{n}$ has been proved. \newline
. Our numerical algorithm has been tested by implementing three test
problems involving a single solitary wave in which analytic solution is
known and expanded it to investigate the interaction of two solitary waves
and evolution of solitons where the analytic solutions are generally unknown
during the interaction. \newline
. The proffered method has been shown to be unconditionally stable. \newline
. For single soliton the $L_{2}$ and $L_{\infty }$ error norms and\ for the
three test problems the invariant quantities $I_{1}$, $I_{2}$ and $I_{3}$
have been computed. From the obtained results it is obviously clear that the
error norms are sufficiently small and the invariants are marginally
constant in all computer run. We can also see that our algorithm for the GEW
equation is more accurate than the other earlier algorithms in the
literature. \newline
. Our method is an effective and a productive method to study behaviors of
the dispersive shallow water waves.


\begin{thebibliography}{99}
\bibitem{mei} Li Q, Mei L. Local momentum-preserving algorithms for the GRLW
equation. Applied Mathematics and Computation 2018;330:77--92.

\bibitem{pereg} Peregrine DH. Calculations of the development of an undular
bore. J. Fluid Mech 1996;25:321--330.

\bibitem{pereg1} Peregrine DH. Long waves on a beach,.J. Fluid Mech
1967;27:815--827.

\bibitem{ben} Benjamin TB, Bona JL, Mahony JJ. Model equations for waves in
nonlinear dispersive systems. Philos. Trans. Royal Soc London
1972;227:47--78.

\bibitem{khalid} Raslan KR, EL-Danaf TS, Ali KK. New numerical treatment for
solving the KDV equation. Journal of Abstract and Computational Mathematics
2017;2(1):1-12.

\bibitem{morrison} Morrison PJ, Meiss JD, Carey JR. Scattering of RLW
solitary waves. Physica 11D 1981:324-336.

\bibitem{hmd} Hamdi S, Enright WH, Schiesser WE, Gottlieb JJ. Exact
solutions of the generalized equal width wave equation, in: Proceedings of
the International Conference on Computational Science and Its Applications
2003;2668:725--734.

\bibitem{sbgk} Karakoc SBG, Zeybek H. A cubic B-spline Galerkin approach for
the numerical simulation of the GEW equation. Stat. Optim. Inf. Comput
2016;4:30--41.

\bibitem{kaya} Kaya D. A numerical simulation of solitary-wave solutions of
the generalized regularized long wave equation. Appl. Math. Comput
2004;149:833--841.

\bibitem{kaya1} Kaya D, El-Sayed SM. An application of the decomposition
method for the generalized KdV and RLW equations. Chaos Solitons Fractals
2003;17:869--877.

\bibitem{gard4} Gardner LRT, Gardner GA, Geyikli T. The boundary forced MKdV
equation. Journal of computational physics 1994;11:5-12.

\bibitem{dodd} Dodd RK, Eilbeck JC, Gibbon JD, Morris HC. Solitons and
Nonlinear Wave Equations. New York: Academic Press; 1982.

\bibitem{levis} Lewis JC, Tjon JA. Resonant production of solitons in the
RLW equation. Phys. Lett. A 1979;73:275-279.

\bibitem{panahipour} Panahipour H. Numerical simulation of GEW equation
using RBF collocation method. Communications in Numerical Analysis 2012;2012:28 pages, doi:10.5899/2012/cna-00059.

\bibitem{gard} Gardner LRT, Gardner GA. Solitary waves of the equal width
wave equation. Journal of Computational Physics 1991;101(1)218--223.

\bibitem{gard1} Gardner LRT, Gardner GA, Ayoup FA, Amein NK. Simulations of
the EW undular bore. Commun. Numer. Meth. En 1997;13:583--592.

\bibitem{zaki} Zaki SI. A least-squares finite element scheme for the EW
equation. Comp. Methods in Appl. Mech. and Eng 2000;189(2)587--594.

\bibitem{esen} Esen A. A numerical solution of the equal width wave equation
by a lumped Galerkin method. Applied Mathematics and Computation
2005;168(1):270--282.

\bibitem{saka} Saka B. A finite element method for equal width equation,
Applied Mathematics and Computation 2006;175(1)730--747.

\bibitem{saka1} Dag I, Saka B. A cubic B-spline collocation method for the
EW equation. Mathematical and Computational Applications 2004; 9(3):381-392.

\bibitem{sbgk1} Karakoc SBGK, Geyikli T. Numerical solution of the modified
equal width wave equation. International Journal of Differential Equations
2012;2012:1--15.

\bibitem{sbgk2} Geyikli T, Karakoc SBG. Petrov-Galerkin method with cubic
B-splines for solving the MEW equation. Bull. Belg. Math. Soc. Simon Stevin
2012;19:215--227.

\bibitem{sbgk3} Geyikli T, Karakoc SBG. Septic B-spline collocation method
for the numerical solution of the modified equal width wave equation. Appl.
Math 2011;2:739--749.

\bibitem{sbgk4} Geyikli T, Karakoc SBG. Subdomain finite element method with
quartic B-splines for the modified equal width wave equation. Computational
Mathematics and Mathematical Physics 2015;55(3):410-421.

\bibitem{sbgk5} Karakoc SBG. Numerical solutions of the modified equal width
wave equation with finite elements method. PhD thesis, Inonu University,
Malatya, Turkey, 2011.

\bibitem{esen1} Esen A. A lumped Galerkin method for the numerical solution
of the modified equal-width wave equation using quadratic B-splines.
International Journal of Computer Mathematics 2006;83(5-6)449--459.

\bibitem{saka2} Saka B. Algorithms for numerical solution of the modified
equal width wave equation using collocation method. Mathematical and
Computer Modelling 2007;45(9-10)1096--1117.

\bibitem{evans} Evans DJ, Raslan KR. Solitary waves for the generalized
equal width (GEW) equation. Int. J. Comput. Math 2005;82(4):445--455.

\bibitem{raslan} Raslan KR. Collocation method using cubic B-spline for the
generalised equal width equation. Int. J. Simulation and Process Modelling
2006;2:37--44.

\bibitem{tag} Taghizadeh N, Mirzazadeh M, Akbari M, Rahimian M. Exact
solutions for generalized equal width equation. Math. Sci. Let
2013;2:99--106.

\bibitem{sbgk6} Zeybek H, Karakoc SBG. Application of the collocation method
with B-splines to the GEW equation. Electronic Transactions on Numerical
Analysis 2017;46:71--88.

\bibitem{roshan} Roshan T. A Petrov--Galerkin method for solving the
generalized regularized equal width (GEW) equation. Journal of Computational
and Applied Mathematics 2011;235:1641-1652.

\bibitem{Noureddine2013} Atouani N, Omrani K. Galerkin finite element method
for the Rosenau-RLW equation. Computers \& Mathematics with Applications
2013;66(3):289-303.

\bibitem{ThomeeVidar2006} Thomee V. Galerkin Finite Element Methods for
Parabolic Problems. Springer Series in Computational Mathematics; ISSN:
0179-3632, second edition, 2006.

\bibitem{ciarlet} Ciarlet PG. The Finite Element Method for Elliptic
Problems. Society for Industrial and Applied Mathematics 2002.

\bibitem{Battel_SKB_2018} Karakoc SBG, Bhowmik SK. Galerkin Finite Element
Solution for Benjamin-Bona-Mahony-Burgers Equation with Cubic B-Splines.
Computers \& Mathematics with Applications. Published, Available online 7
December 2018.

\bibitem{prenter} Prenter P.M. Splines and Variational Methods. John Wiley
\& Sons, New York: NY.USA; 1975.
\end{thebibliography}
\end{document}